\documentclass[11pt,reqno]{amsproc}
\usepackage[margin=1in]{geometry}
\usepackage{amsmath, amsthm, amssymb}
\usepackage[colorlinks=true, pdfstartview=FitV, linkcolor=blue,citecolor=blue, urlcolor=blue]{hyperref}
\usepackage[abbrev,lite,nobysame]{amsrefs}
\usepackage{times}
\usepackage[usenames,dvipsnames]{color}
\usepackage{mathtools}

\mathtoolsset{showonlyrefs=true}

\def\de{{\partial}}
\def\grad{{\nabla}}
\def\Re{{\rm Re}}
\def\d{{\rm d}}
\def\spt{{\rm spt}}
\def\eps{\varepsilon}
\def\e{{\rm e}}
\def \l {\langle}
\def \r {\rangle}
\def\ff {{\widehat f}}
\def\ggg {{\widehat g}}
\def\ddt{{\frac{\d}{\d t}}}
\def\EE {\mathbb{E}}

\def\T {{\mathbb T}}

\def\RR {\mathbb{R}}
\def\ZZ {{\mathbb Z}}
\def\NN {{\mathbb N}}

\newcommand{\Real}{\mathbb R}
\newcommand{\Torus}{\mathbb T}
\newcommand{\Complex}{\mathbb C}

\newcommand{\norm}[1]{\left\lVert#1\right\rVert}
\newcommand{\abs}[1]{\left\vert#1\right\vert}
\newcommand{\set}[1]{\left\{#1\right\}}

\newcommand{\jap}[1]{\left\langle #1 \right\rangle}
\newcommand{\brak}[1]{\left\langle #1 \right\rangle}

\newtheorem{proposition}{Proposition}[section]
\newtheorem{theorem}[proposition]{Theorem}

\newtheorem{lemma}[proposition]{Lemma}
\theoremstyle{definition}

\newtheorem{remark}[proposition]{Remark}

\numberwithin{equation}{section}

\title[Enhanced dissipation, hypoellipticity, and anomalous small noise inviscid limits in shear flows]
{Enhanced dissipation, hypoellipticity, and anomalous small noise inviscid limits in shear flows}

\author[J. Bedrossian and M. Coti Zelati]{Jacob Bedrossian and Michele Coti Zelati}

\address{Department of Mathematics, University of Maryland, College Park, MD 20742, USA}
\email{jacob@cscamm.umd.edu}
\email{micotize@umd.edu}



\begin{document}

\begin{abstract}
We analyze the decay and instant regularization properties of the evolution semigroups generated by
two-dimensional drift-diffusion equations in which the scalar is advected by a shear flow and dissipated 
by full or partial diffusion. 
We consider both the space-periodic $\T^2$ setting and the case of a bounded channel $\T \times [0,1]$ with no-flux boundary conditions.  
In the infinite P\'eclet number limit (diffusivity $\nu\to 0$), our work quantifies the enhanced dissipation effect due to
the shear.
We also obtain hypoelliptic regularization, showing that solutions are instantly Gevrey regular even with only partial diffusion. 
The proofs rely on localized spectral gap inequalities and ideas from hypocoercivity with an augmented energy functional
with weights replaced by pseudo-differential operators (of a rather simple form). 
As an application, we study small noise inviscid limits of invariant measures 
of stochastic perturbations of passive scalars, and show that the classical Freidlin scaling between noise and diffusion
can be modified.  
In particular, although statistically stationary
solutions blow up in $H^1$ in the limit $\nu\to0$, we show that viscous invariant measures still converge to a 
unique inviscid measure.
\end{abstract}


\maketitle

\section{Introduction and main results}

Let $u=u(y):D\to \RR$ be a smooth function, where $D$ denotes either $\T$ or $[0,1]$, and $\nu>0$ a positive parameter. We analyze the decay and instant regularization properties of the linear evolution semigroup $S_\nu(t):L^2(\T\times D)\to L^2(\T\times D)$ generated by the drift-diffusion scalar
equation
\begin{equation}\label{eq:passive}
\begin{cases}
\partial_t f + u \partial_x f = \nu \Delta f, \quad &(x,y)\in \T\times D,\ t>0,\\
f(0,x,y)=f_{in}(x,y),\quad &(x,y)\in \T\times D,
\end{cases}
\end{equation}
and of its \emph{hypoelliptic} counterpart $R_\nu(t):L^2(\T\times D)\to L^2(\T\times D)$, generated by
\begin{equation}\label{eq:hypopassive}
\begin{cases}
\partial_t f + u \partial_x f = \nu \de_{yy} f, \quad &(x,y)\in \T\times D,\ t>0,\\
f(0,x,y)=f_{in}(x,y),\quad &(x,y)\in \T\times D.
\end{cases}
\end{equation}
In the case $D=[0,1]$, equations \eqref{eq:passive}-\eqref{eq:hypopassive} are equipped with the usual no-flux
boundary conditions
\begin{equation}\label{eq:chanpassiveBC}
\de_yf(t,x,0)=\de_yf(t,x,1)=0, \qquad \forall x\in \T,\quad t\geq 0.
\end{equation}
Problem \eqref{eq:passive} belongs to the general class of so-called \emph{passive scalar} 
equations 
$$
\de_tf+\boldsymbol{u}\cdot\grad f=\nu \Delta f,
$$
in the special case when the velocity vector field is a \emph{shear flow}, namely $\boldsymbol{u}(x,y)=(u(y),0)$. 
In this context (after appropriate non-dimensionalization), $\nu^{-1}$ is the P\'eclet number; the dynamics of passive scalars at high P\'eclet number 
is a classical and important topic in applied mathematics and physics; see e.g. \cites{RhinesYoung83,Lundgren82,Bajer2001,LTD11,AT11,S13,IKX14} and the references therein. 
Hypoelliptic problems related to \eqref{eq:hypopassive} arise in the study of boundary layers \cite{LiDiXu15} and in 
kinetic theory \cites{AlexandreEtAl2010, Chen2009320, ChenDesLing2009}.

We assume that $u \in C^{n_0+2}(D)$ has a \emph{finite} number
of critical points, denoted by $\bar y_1,\ldots, \bar y_N$, and where  $n_0\in \NN$ 
denotes the maximal order of vanishing of $u'$ at the critical points, namely, the minimal integer
such that
\begin{equation}\label{eq:last}
u^{(n_0+1)}(\bar y_i)\neq 0, \qquad \forall i=1,\ldots N.
\end{equation}
Finally, we assume without loss of generality that $\int_D u(y) \d y = 0$. 
Notice that in the case of $D=\T$, we necessarily have $n_0\geq 1$, while when $D=[0,1]$ we may consider monotone shear flows.
We first discuss the space-periodic scenario, in which our main result reads as follows.

\begin{theorem}[Periodic $(D=\T)$ case] \label{thm:maindecay}
There exist positive constants $\eps\ll1$, $C\geq 1$ and $\kappa_0\ll1$ (depending only on $u$) such that 
the following holds: for every $\nu>0$ and every integer $k\neq 0$ such that 
\begin{equation}\label{eq:kappa}
\nu |k|^{-1}\leq \kappa_0,
\end{equation}
there hold the $L^2$ decay estimates 
\begin{align}
\norm{S_\nu(t)P_k}_{L^2 \rightarrow L^2} & \leq C\e^{-\eps\lambda_{\nu,k} t}, \qquad \forall t\geq 0,\label{eq:maindecay}
\end{align}
and
\begin{align}
\norm{R_\nu(t)P_k}_{L^2 \rightarrow L^2} & \leq C\e^{-\eps\lambda_{\nu,k} t}, \qquad \forall t\geq 0, \label{eq:hypomaindecay}
\end{align}
where $P_k$ denotes the projection to the $k$-th Fourier mode in $x$ and 
\begin{align}\label{eq:decayrate}
\lambda_{\nu,k} = \frac{\nu^{\frac{n_0+1}{n_0+3}}\abs{k}^{\frac{2}{n_0+3}}}{(1 + \log \abs{k} + \log \nu^{-1})^2}
\end{align}
is the decay rate.
\end{theorem}  

\begin{remark}[Hypoellipticity] \label{rmk:Hypo}
Notice that since the smallness condition \eqref{eq:kappa} is only on $\nu \abs{k}^{-1}$, Theorem \ref{thm:maindecay} is still quite meaningful in the case $\nu = 1$, at least when applied to $R_\nu(t)$ and considering the asymptotic $|k|\to \infty$.
Specifically, the estimate \eqref{eq:hypomaindecay} (along with the standard parabolic smoothing in $y$) shows that data which are only initially $L^2$ become instantly infinitely smooth, and even Gevrey-$p$ regular for all $p > \frac{n_0+ 3}{2}$, where Gevrey-$p$, $\mathcal{G}^{p}$ with $p \geq 1$, is defined via (introduced in \cite{Gevrey18}),   
\begin{align*}
\mathcal{G}^p = \bigcup_{\lambda > 0}\set{f \in L^2 : \norm{\e^{\lambda \abs{\grad}^{1/p}} f}_{L^2} < \infty}. 
\end{align*}
Additionally, \eqref{eq:hypomaindecay} provides a quantitative estimate of the Gevrey regularity in $x$. 
We conjecture that \eqref{eq:hypomaindecay} is sharp in terms of \emph{relative} scaling between $k$ and $t$,
up to the logarithmic corrections. For instance, the optimal Sobolev regularization estimate for the solution $f$ 
of \eqref{eq:hypopassive} would read
$$
\|\de_x^\alpha f(t)\|_{L^2}\leq \frac{C}{t^{\frac{\alpha(n_0+3)}{2}}} \| f_{in}\|_{L^2}.
$$
Moreover, the scaling in $|k|$ relative to $t$ is in agreement with the existing known cases $u(y)=y$ or $u(y)=y^2$ (see \cites{Kelvin87,BZ09, Beauchard2014}) and is consistent with the mixing time-scales deduced in Appendix \ref{app:mixing} for the
inviscid ($\nu=0$) counterpart of \eqref{eq:hypopassive}.

This instant smoothing effect, despite not having any dissipative mechanisms in $x$, is an example of hypoellipticity \cites{Hormander67,Hormander1985III}. 
Finally, we remark that Gevrey class hypoelliptic smoothing has been identified in the Prandtl equations \cite{LiDiXu15} and in collisional kinetic theory for the Boltzmann (without angular cut-off) \cite{AlexandreEtAl2010} and the Landau-Fokker-Planck equations \cite{ChenDesLing2009}.  
\end{remark}

Let us first discuss the intuition behind the scaling of $\lambda_{\nu,k}$ in \eqref{eq:decayrate}. 
The inviscid problem 
$$
\partial_t f + u\partial_x f = 0,
$$ 
can be shown to satisfy the following $H^{-1}$ decay estimate in the $y$-variable via the method of stationary phase (see Appendix \ref{app:mixing}), 
\begin{align}
\norm{P_k f(t)}_{H^{-1}_y} \lesssim \jap{kt}^{-\frac{1}{n_0 + 1}} \norm{P_k f(0)}_{H^1_y}. \label{ineq:invdec}
\end{align}
The $H^{-1}$ norm is often used to quantify mixing \cites{LTD11,IKX14} as it provides a kind of 
averaged measure of the characteristic length-scale of the solution. 
In particular, the estimate \eqref{ineq:invdec} suggests that $P_kf(y)$ is concentrated in 
frequencies $\eta$ which satisfy $\abs{\eta} \gtrsim \jap{kt}^{\frac{1}{n_0+1}}$. 
For small $\nu$, we expect the inviscid mixing to be the leading order dynamics (at least 
for some time), and hence we can expect the dissipation $\nu \partial_{yy}$ 
to behave on the frequency side like $-\nu \jap{kt}^{\frac{2}{n_0 + 1}} \ff$ damping. Upon 
integration, this predicts a time-scale in agreement with \eqref{eq:decayrate} 
up to the logarithms. 
Notice that the hypoelliptic regularization discussed in Remark \ref{rmk:Hypo} is due to the 
fact that the decay rate in \eqref{ineq:invdec} is faster for large $k$.  

Theorem \ref{thm:maindecay} above describes in a precise manner the so-called \emph{enhanced dissipation} effects caused by the shear flow. 
While the natural dissipation time-scale is that of 
the heat equation, i.e. $O(\nu^{-1})$, the mixing due to the shear flow induces a faster time-scale $O(\nu^{-p})$, 
with $p=p(n_0)\in (0,1)$ as in \eqref{eq:decayrate}, at which the $L^2$ density is dissipated. 
This effect in linear equations has been studied previously in \cites{CKRZ08,Zlatos2010,BeckWayne11,VukadinovicEtAl2015,BCZGH15}, in the physics literature \cites{Lundgren82,RhinesYoung83,DubrulleNazarenko94,LatiniBernoff01,BernoffLingevitch94}, and in control theory \cites{BZ09, Beauchard2014}; a closely related effect was also studied in \cite{GallagherGallayNier2009}. 
The effect has turned out also to be important for understanding stability of the Couette flow in the Navier-Stokes equations at high Reynolds number \cites{BMV14,BGM15I,BGM15II}.  
A number of works also focus on estimating the spectrum of the elliptic operator $\boldsymbol{u} \cdot \grad - \nu \Delta$, in particular, see \cites{BHN05,VukadinovicEtAl2015} and  the references therein. However, as the operator is non-normal, precise estimates on the spectrum do not necessarily yield optimal $L^2$ decay estimates of the semigroup in the singular limit $\nu \rightarrow 0$, indeed, one needs estimates on the \emph{pseudo-spectrum}; see Section \ref{sec:L2decay} and \cites{Trefethen2005,GallagherGallayNier2009} for more discussion.
Nonetheless, the scaling of the spectral gap as a power of $\nu$ observed in \cite{VukadinovicEtAl2015} is consistent with Theorem \ref{thm:maindecay} (up to the logarithms). 
In our work, we estimate the pseudo-spectrum using a hypocoercivity method related to the methods of \cites{GallagherGallayNier2009, BeckWayne11} (see below for more discussion). 
For the mixing of passive scalars by shear flows at high Peclet number ($\nu \rightarrow 0$), Theorem \ref{thm:maindecay} provides the most precise, quantitative estimates to date for a general class of shears. 
We conjecture that the rate \eqref{eq:decayrate} is sharp upon removal of the logarithmic losses (see Section \ref{sec:L2decay}). 

In the case $u(y)=y$, the enhanced dissipation effect for \eqref{eq:passive}  was first deduced by
Kelvin \cite{Kelvin87}, whereas the hypoellipticity of \eqref{eq:hypopassive} was first considered by 
Kolmogorov \cite{Kolmogorov34}. 
For more general flows $u$, it is easy to check that equation \eqref{eq:hypopassive}
satisfies H\"ormander's bracket condition, provided the conditions of Theorem \ref{thm:maindecay} on $u$ are satisfied (that is, finitely many critical points and all vanish only to finite order), and hence \eqref{eq:hypopassive} falls under H\"ormander's general theory of hypoelliptic operators \cites{Hormander67, Hormander1985III}. 
We provide more quantitative Gevrey hypoellipticity as well as the long-time enhanced dissipation estimate, indeed, our work suggests that the two concepts are intimately connected and are here  deduced simultaneously. 

The proof of Theorem \ref{thm:maindecay} relies on ideas connected with the hypocoercivity of the linear drift-diffusion operator appearing in \eqref{eq:passive}-\eqref{eq:hypopassive}.
Hypocoercivity originated in kinetic theory to study the long-time behavior of collisional models, see e.g. \cites{DesvillettesVillani01,DMS15,HerauNier2004,Herau2007} and the references therein or the text \cite{villani2009} for an overview of the basic ideas and a general study of hypocoercivity applied to linear operators in H\"ormander form.      
These techniques are centered around finding an augmented energy function which is essentially a weighted sum of commutators of the symmetric and skew-symmetric terms in the linear operator.  
The operator $u\de_x-\nu\Delta$ constitutes a very particular case of the general concepts \cite{villani2009}, however, it requires significant additional care to study the singular limits $\nu \to 0$ and $\abs{k} \rightarrow \infty$ which allow to identify enhanced dissipation and hypoellipticity. 
The work \cite{GallagherGallayNier2009} studies the spectral and pseudo-spectral properties (see \cite{Trefethen2005} or Section \ref{sec:L2decay}) of the harmonic oscillator with an additional large, complex potential.   
Therein the authors demonstrated the potential usefulness of hypocoercivity in studying such limits as well as introducing the idea of using weighted norms adapted to features in the skew-symmetric operator.
The authors of \cite{BeckWayne11} expanded on some of the ideas from \cite{GallagherGallayNier2009} to study enhanced dissipation near the particular, time-varying shear flow $u(t,y) = \e^{-\nu t}\sin \pi y$ (and so they do not obtain a semigroup decay estimate as in Theorem \ref{thm:maindecay}). 
Therein, the authors introduced the idea of making the weighted sums in the augmented energy functional $k$-dependent, effectively using Fourier multipliers instead of constants as weights, although a careful study of the limit $k \rightarrow \infty$ and Gevrey regularization is not explicit. 
Some similar ideas also appeared in \cite{Beauchard2014}, however, there the authors restricted attention to the much simpler cases of $u(y) = y$ or $u(y) = y^2$.  

In our work, we use an augmented energy functional with coefficients that are pseudo-differential operators, though of the simple form of $y$-dependent Fourier multipliers in $x$ (see Section \ref{sec:abgdef}). 
This is essential to deal carefully with the critical points of the shear flow or the presence of boundaries in Theorem \ref{thm:channeldecay} below. 
Notice that after taking a Fourier transform in the $x$-direction, the problem decouples in frequency and the decay estimate \eqref{eq:maindecay} is $k$-by-$k$. 
This does not only make the estimate more precise -- it seems to be unavoidable.
With some extra precision, we are able to treat the degenerate drift-diffusion operator $u\de_x-\nu\de_{yy}$ as a by-product of our analysis and obtain the hypoellipticity discussed in Remark \ref{rmk:Hypo}. 
The connection between hypocoercivity and hypoellipticity is the origin of the name of the former \cite{villani2009}, however, Theorem \ref{thm:maindecay} and the proof emphasize that in certain singular limits, they are mathematically \emph{equivalent} and together give rise to the enhanced dissipation phenomenon. 

Our work also applies to the case of the bounded channel $\Torus \times [0,1]$ with no-flux boundary conditions \eqref{eq:chanpassiveBC} imposed. 
In particular, we deduce the following theorem. 
 
\begin{theorem}[Channel case]\label{thm:channeldecay}
There exist positive constants $\eps\ll1$, $C\geq 1$ and $\kappa_0\ll1$ (depending only on $u$) such that 
the following holds: for every $\nu>0$ and every integer $k\neq 0$ such that 
\begin{equation}\label{eq:kappabdy}
\nu |k|^{-1}\leq \kappa_0,
\end{equation}
there hold the $L^2$ decay estimates 
\begin{align}
\norm{S_\nu(t)P_k}_{L^2 \rightarrow L^2} & \leq C\e^{-\eps\lambda_{\nu,k} t}, \label{eq:maindecaybdy}
\end{align}
and
\begin{align}
\norm{R_\nu(t)P_k}_{L^2 \rightarrow L^2} & \leq C\e^{-\eps\lambda_{\nu,k} t}, \label{eq:hypomaindecaybdy}
\end{align}
where $P_k$ denotes the projection to the $k$-th Fourier mode in $x$ and 
\begin{align*}
\lambda_{\nu,k} = \frac{\nu^{\frac{n_c+1}{n_c+3}}\abs{k}^{\frac{2}{n_c+3}}}{(1 + \log \abs{k} + \log \nu^{-1})^2}, 
\end{align*}
where  $n_c=\max\{n_0,1\}$.
\end{theorem}
In the case of strictly monotone shear flows, it is understood that $n_0=0$ in the statement above. 
Compared to the decay rate devised in Theorem \ref{thm:maindecay}, the presence of the boundary 
introduces an additional effect that is comparable to the presence of a critical point of vanishing order 1,
at least from the scaling in $\nu$ and $k$ of the rate. We are currently unsure if this is optimal, however, a heuristic reason for why it might be is
as follows: the no-flux boundary conditions on $f$ allow to study the problem via even reflection around $y=1$, effectively introducing a fictitious critical point at $y=1$ and hence degrading the decay rate (and similarly to study $y = 0$). 
We also remark that no-flux boundaries can cause issues when studying inviscid damping in the linearized Euler equations \cite{Zillinger2015}. 

\begin{remark}[Extensions to other spatial settings] \label{rmk:exten}
Theorem \ref{thm:maindecay} can be easily extended to a few additional settings: 

\medskip

\noindent $\diamond$ $\T\times [0,\infty)$ or $\Torus \times \Real$: Assuming that there is some $c > 0$ such that $\abs{u'(y)} \geq c$ outside of a compact set, the analysis we employ applies also to geometries which are unbounded in $y$ via some modifications of the proof (at least in the elliptic case).

\noindent $\diamond$ $\RR \times D$: More non-trivially, the proof of Theorem \ref{thm:maindecay} does \emph{not} use any specific condition on $k$ or $\nu$ besides $\nu \abs{k}^{-1} \leq \kappa_0$, and hence immediately applies to tori of general side-length $L$, $\T_L \times D$, and also to $\RR \times D$, for those frequencies which satisfy $\nu \leq \kappa_0 \abs{k}$, holding pointwise in $k$. 
\end{remark}

\subsection{Small noise inviscid limits and anomalous scalings}
An interesting application of our semigroup estimate is related to small noise inviscid limits of stochastically forced
passive scalars, in the context of invariant measures. Consider equation \eqref{eq:passive}, posed in the
two-dimensional torus and stochastically forced, namely
\begin{equation}\label{eq:viscostoc}
\d f + \left(u\de_x f - \nu \Delta f\right) \d t = \nu^{a/2} \, \Psi\, \d W_t,
\end{equation} 
where $a>0$ is a fixed parameter and
\begin{align*}
\Psi \d W_t = \sum_{(k,j) \in \ZZ^2} \psi_{k,j} e_{k,j}\d W^{k,j}_t, \qquad \psi_{k,j} = \overline{\psi_{-k,-j}}, \qquad W^{k,j}_t  = W^{-k,-j}_t,
\end{align*} 
denotes white-in-time, colored-in-space noise  defined through the standard Fourier basis
\begin{align*}
e_{k,j} = \frac{1}{4 \pi^2} \e^{-ikx - ijy} 
\end{align*}
and one-dimensional independent Brownian motions $\{W^{k,j}_t\}_{(k,j) \in \ZZ^2}$.
It is a classical result \cite{DZ96} that under the assumption
$$
\|\Psi\|^2 =\sum_{(k,j) \in \ZZ^2} |\psi_{k,j}|^2 <\infty,
$$
the Markov semigroup generated by \eqref{eq:viscostoc} has a unique invariant 
Gaussian measure $\mu_\nu=\mathcal{N}(0,Q_\nu)$ on $L^2$, with covariance
operator given by
\begin{equation}\label{eq:convar}
Q_\nu = \nu^{a}\int_0^\infty S_\nu(t)\Psi \Psi^\ast S_\nu(t)^\ast \d t. 
\end{equation}
As $\nu\to 0$, we obtain the following behavior for the sequence $\{\mu_\nu\}_{\nu\in (0,1]}$.
\begin{theorem}\label{thm:kuksin}
Within the assumption of Theorem \ref{thm:maindecay}, suppose 
that $\psi_{0,j}=0$ for all $j\in \ZZ$,  and let the parameter $a$ satisfy
\begin{equation}\label{eq:aa}
a\in \left(\frac{n_0+1}{n_0+3},1\right].
\end{equation}
Then, as $\nu\to 0$, we have that $\mu_\nu\to \delta_0$, a Dirac mass centered at 0, in the sense of measures. 
\end{theorem}
\begin{remark} 
In fact, the convergence we deduce is slightly more precise: we show that the covariance \eqref{eq:convar} vanishes in the strong operator topology.
\end{remark}
In the case $a=1$ (the so-called Freidlin scaling \cites{Freidlin2002, FW1994}), 
this type of inviscid limit has been recently analyzed in \cite{BCZGH15},
without the requirement that $\psi_{0,j}=0$ for all $j\in \ZZ$. More generally, 
the idea of balancing noise and diffusivity has been considered
extensively in the context of two-dimensional Navier-Stokes equations and other nonlinear settings \cites{K04,K10,KS04,KS12,GSV13,MP14}. In all these works, the requirement $a=1$
is dictated by the following reason: a simple application of It\^o's formula implies the energy balance
\begin{equation}\label{eq:energy}
	\EE \|f^\nu(t)\|^2_{L^2}+2\nu \EE\int_0^t\|f^\nu(s)\|^2_{H^1}\d s=
	\EE \|f^\nu(0)\|^2_{L^2}+\nu^a\|\Psi\|^2t, \qquad \forall t\geq0.
\end{equation}
In particular, any statistically stationary (with respect to the invariant measure $\mu_\nu$) solution $f^\nu_S$  satisfies
\begin{equation}\label{eq:energy2}
 \EE\|f_S^\nu(t)\|^2_{H^1}=\frac{\nu^{a-1}}{2}\|\Psi\|^2, \qquad \forall t\geq 0.
\end{equation}
Therefore, if $a=1$ we have a  \emph{uniform}  $H^1$-bound that translates into the
compactness of $\{\mu_\nu\}_{\nu\in (0,1]}$,
and hence existence of subsequential limit points. Such limiting measures, sometimes 
referred to as \emph{Kuksin} measures (due to \cite{K04}), can be proved to be invariant for the inviscid equation
\begin{equation}\label{eq:inviscid}
\de_t f + u \de_x f = 0.
\end{equation}
It is therefore evident that when $a<1$, the above procedure cannot be applied since statistically stationary solutions blow up in $H^1$ by \eqref{eq:energy2}. 
In fact, in the context of the two-dimensional Navier-Stokes equation, when $a<1$ it can be proved 
that  $\{\mu_\nu\}_{\nu\in (0,1]}$ has no accumulation points in the sense of weak
convergence on $L^2$ (c.f. \cite{KS12}*{Theorem 5.2.17}). 
In the situation described by Theorem \ref{thm:kuksin}, the picture is therefore strikingly different and we cannot rely on soft compactness arguments. 
Rather, we make use of the explicit covariance formula \eqref{eq:convar}, available in the linear setting, together with the explicit semigroup 
decay estimate \eqref{eq:maindecay}, which allows the relaxation of the constraint on $a$ as stated in \eqref{eq:aa}. 
Note, Theorem \ref{thm:kuksin} also proves the existence of a \emph{unique} Kuksin measure, which is not known in most linear or nonlinear settings.  

\subsection{Outline of the article}
Section \ref{sec:noidea} is dedicated to the construction of a suitable energy functional $\Phi$ with exponential
decay rate analogous to \eqref{eq:decayrate}. We also prove a localized spectral gap inequality that allows
to treat each individual critical point in a sharp way. In Section \ref{sec:err} we estimate various error terms appearing in the
differential inequality for $\Phi$ and prove that in fact $\Phi$ decay exponentially to zero,
while in Section \ref{sec:L2decay} we show how this translates into the $L^2$-decay estimate 
\eqref{eq:maindecay} for $S_\nu(t)$. The proof of Theorem \ref{thm:channeldecay} is carried out in Section \ref{sec:chan},
where we adapt the machinery to the case of the channel. Section \ref{sec:kuksin} is devoted to the proof of the
small noise inviscid limit result stated in Theorem \ref{thm:kuksin}. We conclude the article with Appendix \ref{app:mixing},
in which we prove a decay estimate on the inviscid problem \eqref{eq:inviscid}.

\section{Augmented energy functional}\label{sec:noidea}
In what follows, we will be working with the the solution operator 
$S_\nu(t)$ of \eqref{eq:passive}, and will highlight the differences with the 
other equations when needed.
As discussed in the introduction, the statements of Theorems \ref{thm:maindecay}-\ref{thm:channeldecay}
provide estimates which involve the \emph{non-zero} frequencies in the $x$-direction. 
In physical variables, it is therefore natural to assume
$$
\int_\T f_{in}(x,y)\d x=0,\qquad \forall y\in \T,
$$
a property that is preserved by the equation, so that
\begin{equation}\label{eq:zeromean1}
\int_\T f(t,x,y)\d x=0,\qquad \forall y\in \T,\quad t>0.
\end{equation}
Equivalently, we expand the solution to \eqref{eq:passive}
as
$$
f(t,x,y)=\sum_{k\neq 0} \ff_k(t,y)\e^{ikx},
$$
where the frequency $k=0$ is neglected since \eqref{eq:zeromean1} immediately translates into
\begin{equation}\label{eq:zeromean2}
\ff_0(t,y)=0,\qquad \forall y\in \T,\quad t>0.
\end{equation}
By taking a Fourier transform in $x$, equation \eqref{eq:passive}
becomes
\begin{equation}\label{eq:fourier1}
\de_t \ff_k+iku\ff_k=\nu \big[-|k|^2+\de_{yy}\big] \ff_k, \qquad \ff_k(0,y)=\widehat{f_{in}}_{k}(y).
\end{equation}
Notice that both equations decouple in $k$, so that we can consider each frequency separately. For this
reason, throughout the remaining of the article, we will suppress every dependence on $k\neq 0$ to avoid
further complications in the notation. We will denote by $L^2$ the one-dimensional 
Lebesgue space of square integrable $y$-dependent functions on $\T$, endowed with the 
scalar product
$$
\l \varphi,\psi\r=\int_\T \varphi(y)\overline{\psi(y)}\d y
$$
and norm
$$
\|\varphi\|=\left[\int_\T |\varphi(y)|^2\d y\right]^{1/2}.
$$
As for other normed spaces, their norms will be explicitly indicated as $\|\cdot\|_X$.
The main step in the proof of Theorem \ref{thm:maindecay} is the following theorem. 
\begin{theorem}\label{thm:decay1}
There exists a positive constant $\kappa_0\ll1$ such that the following holds: for each integer $k\neq 0$ and $\nu>0$ with
$\nu |k|^{-1}\leq \kappa_0$, 
there exist a positive functional $\Phi(t)$ so that:
\begin{enumerate}
	\item $\Phi$ is comparable to the $H^1$-norm, namely
	\begin{equation}\label{eq:prop1}
	\frac{\nu^{2/3}}{|k|^{2/3}}\|\de_y\ff\|^2+ \|u'\ff\|^2
	\lesssim \Phi-\|\ff\|^2\lesssim \|\de_y\ff\|^2+\frac{|k|^\frac{2n_0}{n_0+3}}{\nu^{\frac{2n_0}{n_0+3}}}\|u'\ff\|^2.
	\end{equation}
	\item $\Phi$ satisfies the decay estimate
	\begin{equation}\label{eq:prop2}
	\Phi(t)\leq  \Phi(0) \e^{-\widetilde\eps\widetilde\lambda_{\nu,k}t}, \qquad \forall t\geq 0,
	\end{equation}
	where
	\begin{equation}\label{eq:prop3}
	\widetilde\lambda_{\nu,k}=\, \nu^\frac{n_0+1}{n_0+3}|k|^\frac{2}{n_0+3},
	\end{equation}
	and $\widetilde\eps>0$ is a constant independent of $\nu$ and $k$. 
\end{enumerate}
\end{theorem}

Roughly speaking, the basic idea of hypocoercivity is to find an equivalent norm (as expressed in e.g. \eqref{eq:prop1}) for which the linear operator is coercive (as expressed in \eqref{eq:prop2}); see \cite{villani2009} for more discussion. 
However, the situation here is more subtle than in \cite{villani2009} due to the interest in the singular limit $\nu \abs{k}^{-1} \rightarrow 0$, which requires more sophisticated techniques and a more precise $\Phi$, as was observed already in \cites{BeckWayne11,GallagherGallayNier2009}; the proof of Theorem \ref{thm:maindecay} will require even more precision. 

\subsection{The energy functional}
Let $\alpha,\beta,\gamma$ be $y$-dependent positive smooth functions defined for $y\in \T$, whose precise
form will be specified below (see Section \ref{sec:abgdef}). For each integer $k\neq 0$, we define the energy-like functional
\begin{equation}\label{eq:Phi}
\Phi(t)=\|\ff(t)\|^2+\|\sqrt\alpha\de_y \ff(t)\|^2+2 k \Re \l i\beta u'  \ff(t), \de_y\ff(t)\r+ |k|^2\|\sqrt\gamma u' \ff(t)\|^2.
\end{equation}
In this section and the next one, we aim to prove that $\Phi$ has the properties stated in Theorem \ref{thm:decay1}.
The proof relies on a number of intermediate lemmas and will be carried out below. As a preliminary
step, we impose the pointwise constraint
\begin{equation}\label{eq:constr1}
\frac{\beta(y)^2}{\alpha(y)}\leq\frac18 \gamma(y),\qquad  \forall y\in \T.
\end{equation}
In this way, the cross-term involving 
$\beta$ in \eqref{eq:Phi} can be estimated as
\begin{align*}
2 k \Re \l i\beta u'  \ff, \de_y\ff\r
&\leq 2|k|\left\|\frac{\beta}{\sqrt\alpha}u' \ff\right\|\|\sqrt\alpha\de_y\ff\|\\
&\leq\frac12 \|\sqrt\alpha\de_y\ff\|^2+4 |k|^2\left\|\frac{\beta}{\sqrt\alpha}u' \ff\right\|^2\\
&\leq\frac12 \|\sqrt\alpha\de_y\ff\|^2+\frac12 |k|^2\|\sqrt\gamma u' \ff\|^2.
\end{align*}
Consequently,
\begin{equation}\label{eq:equiv1}
\|\ff\|^2+\frac12\|\sqrt\alpha\de_y \ff\|^2+ \frac12 |k|^2\|\sqrt\gamma u' \ff\|^2\leq\Phi
\leq \|\ff\|^2+\frac32\|\sqrt\alpha\de_y \ff\|^2+ \frac32 |k|^2\|\sqrt\gamma u' \ff\|^2.
\end{equation}
The claim in Theorem \ref{thm:decay1} relies on the derivation of a suitable differential inequality for $\Phi$. 

\begin{lemma}
Let $\Phi$ be defined as in \eqref{eq:Phi}. Then
\begin{equation}\label{eq:der}
\begin{aligned}
\ddt \Phi =&-2\nu |k|^2\|\ff\|^2-2\nu\|\de_y\ff\|^2\\
&-2\nu|k|^2\|\sqrt\alpha\de_y \ff\|^2-2\nu\|\sqrt\alpha\de_{yy}\ff\|^2
-2 k\Re \l i \alpha u'\ff,\de_y\ff\r-2\nu\Re\l\alpha'\de_y\ff,\de_{yy}\ff\r\\
&-2|k|^2\|\sqrt\beta u'\ff\|^2-4\nu k^3\Re\l i \beta u'\ff,\de_y\ff\r+ 4\nu k\Re \l i \beta u'\de_{yy}\ff,\de_y\ff\r
\\
&+2\nu k\Re\l i\beta u'''\ff,\de_y\ff\r+4\nu k \Re\l i\beta'u''\ff,\de_y\ff\r+2\nu k \Re \l i\beta''u'\ff,\de_y\ff\r\\ 
&-2\nu|k|^4\|\sqrt\gamma u'f\|^2-2\nu|k|^2\|\sqrt\gamma u'\de_y\ff\|^2-4\nu|k|^2\Re\l \gamma u'u''\ff,\de_y\ff\r\\
&-2\nu|k|^2\Re\l \gamma'u'\ff,u'\de_y\ff\r.
\end{aligned}
\end{equation}
\end{lemma}

\begin{proof}
The proof essentially relies on  \eqref{eq:fourier1}, several integration by parts and the antisymmetric 
properties of the transport term. We give a few details, as it will
be useful to highlight where boundaries will play a role. Trivially,
\begin{equation}\label{eq:der1}
\ddt\|\ff\|^2=2\Re\l\de_t\ff,\ff\r=-2\nu |k|^2\|\ff\|^2-2\nu\|\de_y\ff\|^2,
\end{equation}
and
\begin{align}
\ddt\|\sqrt\alpha\de_y \ff\|^2&=2\Re\l\alpha\de_{y}\ff,\de_{yt}\ff\r\nonumber\\
&=-2k\Re\l i\alpha u'\ff,\de_y\ff\r-2\nu|k|^2\|\sqrt\alpha \de_y\ff\|^2+2\nu\Re\l\alpha \de_{yyy}\ff,\de_y\ff\r\nonumber\\
&=-2k\Re\l i\alpha u'\ff,\de_y\ff\r-2\nu|k|^2\|\sqrt\alpha \de_y\ff\|^2-2\nu\|\sqrt\alpha \de_{yy}\ff\|^2
-2\nu\Re\l\alpha' \de_{y}\ff,\de_{yy}\ff\r.\label{eq:der2}
\end{align}
The term containing $\beta$ is slightly more complicated. Firstly notice that
\begin{align*}
\ddt  \l i\beta u'  \ff, \de_y\ff\r&= \l i\beta u'  \de_t\ff, \de_y\ff\r+\l i\beta u'  \ff, \de_{yt}\ff\r\\
&= -2\nu|k|^2\l i \beta u'\ff,\de_y\ff\r-k\|\sqrt\beta u'\ff\|^2+\nu\l i\beta u' \de_{yy}\ff,\de_y\ff\r+\nu\l i\beta u'\ff,\de_{yyy}\ff\r.
\end{align*}
We integrate by parts in the last term to deduce 
\begin{equation}\label{eq:nontrivbdry}
\l i\beta u'\ff,\de_{yyy}\ff\r= -\l i\beta' u'\ff,\de_{yy}\ff\r-\l i\beta u''\ff,\de_{yy}\ff\r-\l i\beta u'\de_y\ff,\de_{yy}\ff\r,
\end{equation}
and therefore
\begin{align*}
\ddt \l i\beta u'  \ff, \de_y\ff\r=
&-2\nu|k|^2\l i \beta u'\ff,\de_y\ff\r-k\|\sqrt\beta u'\ff\|^2+\nu\Re\l i\beta u' \de_{yy}\ff,\de_y\ff\r\\
&-\nu\l i\beta' u'\ff,\de_{yy}\ff\r-\nu\l i\beta u''\ff,\de_{yy}\ff\r-\nu\l i\beta u'\de_y\ff,\de_{yy}\ff\r.
\end{align*}
We will have to take the real part of the above expression. By
integration by parts and anti-symmetry, we can compute
$$
-\Re\l i\beta' u'\ff,\de_{yy}\ff\r-\Re\l i\beta u''\ff,\de_{yy}\ff\r=\Re\l i\beta'' u'\ff,\de_{y}\ff\r+2\Re\l i\beta' u''\ff,\de_{y}\ff\r
+\Re\l i\beta u'''\ff,\de_{y}\ff\r,
$$
to finally obtain
\begin{equation}\label{eq:der3}
\begin{aligned}
\ddt \left[2 k \Re \l i\beta u'  \ff, \de_y\ff\r\right]=
&-2|k|^2\|\sqrt\beta u'\ff\|^2-4\nu k^3\Re\l i \beta u'\ff,\de_y\ff\r+ 4\nu k\Re \l i \beta u'\de_{yy}\ff,\de_y\ff\r\\
&+2\nu k\Re\l i\beta u'''\ff,\de_y\ff\r+4\nu k \Re\l i\beta'u''\ff,\de_y\ff\r+2\nu k \Re \l i\beta''u'\ff,\de_y\ff\r.
\end{aligned}
\end{equation}
As for the last term containing $\gamma$ we have 
\begin{align*}
\ddt\|\sqrt\gamma u' \ff\|^2&=2\Re\l \gamma (u')^2\ff,\de_t\ff\r 
=-2\nu|k|^2\|\sqrt\gamma u'f\|^2+2\nu\Re\l\gamma(u')^2\ff,\de_{yy}\ff\r\\
&=-2\nu|k|^2\|\sqrt\gamma u'f\|^2-2\nu\|\sqrt\gamma u'\de_y\ff\|^2-4\nu\Re\l \gamma u'u''\ff,\de_y\ff\r-2\nu\Re\l \gamma'u'\ff,u'\de_y\ff\r,
\end{align*}
so that
\begin{equation}\label{eq:der4}
\begin{aligned}
\ddt |k|^2\|\sqrt\gamma u' \ff\|^2=&-2\nu|k|^4\|\sqrt\gamma u'f\|^2-2\nu|k|^2\|\sqrt\gamma u'\de_y\ff\|^2\\
&-4\nu|k|^2\Re\l \gamma u'u''\ff,\de_y\ff\r-2\nu|k|^2\Re\l \gamma'u'\ff,u'\de_y\ff\r.
\end{aligned}
\end{equation}
Adding each term above, we reach the desired conclusion.
\end{proof}
Before we proceed, we highlight the differences with the hypoelliptic and the channel cases.

\begin{remark}[The hypoelliptic case]
In the case when $f$ satisfies equation \eqref{eq:hypopassive}, without diffusion in $x$, its $x$-Fourier transform,
still denoted by $\ff$, satisfies
\begin{equation}\label{eq:fourier2}
\de_t \ff+iku\ff=\nu\de_{yy} \ff, \qquad \ff(0,y)=\widehat{f_{in}}(y).
\end{equation}
As a consequence, a few terms are not
present in the equation for the time derivative of $\Phi$. Specifically, the terms with coefficient $|k|^2$ disappear
in \eqref{eq:der1}-\eqref{eq:der2}, the term with $k^3$ in \eqref{eq:der3}, and the term with $|k|^4$ in \eqref{eq:der4}.
It is important to notice, that none of these negative terms are used in the analysis below (see the statements of the lemmas
in Section \ref{sec:err}). The only instance in which the negative term 
$$
-2\nu|k|^4\|\sqrt\gamma u'f\|^2
$$ 
is used is in \eqref{eq:hypoerr} to control 
$$
4\nu k^3\Re\l i \beta u'\ff,\de_y\ff\r,
$$
which disappears in the hypoelliptic case.
In this case, the derivative of $\Phi$ satisfies the equation
\begin{equation}\label{eq:hypoder}
\begin{aligned}
\ddt \Phi =&-2\nu\|\de_y\ff\|^2\\
&-2\nu\|\sqrt\alpha\de_{yy}\ff\|^2
-2 k\Re \l i \alpha u'\ff,\de_y\ff\r-2\nu\Re\l\alpha'\de_y\ff,\de_{yy}\ff\r\\
&-2|k|^2\|\sqrt\beta u'\ff\|^2+ 4\nu k\Re \l i \beta u'\de_{yy}\ff,\de_y\ff\r\\
&+2\nu k\Re\l i\beta u'''\ff,\de_y\ff\r+4\nu k \Re\l i\beta'u''\ff,\de_y\ff\r+2\nu k \Re \l i\beta''u'\ff,\de_y\ff\r\\ 
&-2\nu|k|^2\|\sqrt\gamma u'\de_y\ff\|^2-4\nu|k|^2\Re\l \gamma u'u''\ff,\de_y\ff\r\\
&-2\nu|k|^2\Re\l \gamma'u'\ff,u'\de_y\ff\r.
\end{aligned}
\end{equation}
\end{remark}

\begin{remark}[The channel case] \label{rmk:Channel}
On the channel $\T\times [0,1]$ with the no-flux boundary conditions \eqref{eq:chanpassiveBC},
it is not hard to check that the only boundary terms arises in \eqref{eq:nontrivbdry} in the form
\begin{equation}\label{eq:nontrivbdry2}
\l i\beta u'\ff,\de_{yyy}\ff\r= -\l i\beta' u'\ff,\de_{yy}\ff\r-\l i\beta u''\ff,\de_{yy}\ff\r-\l i\beta u'\de_y\ff,\de_{yy}\ff\r
+\left[i\beta u'\ff\de_{yy}\ff\,\right]\Big|_{y=0}^1,
\end{equation}
producing therefore an additional term in the time derivative of $\Phi$:
\begin{equation}\label{eq:bdryder}
\begin{aligned}
\ddt \Phi =&-2\nu |k|^2\|\ff\|^2-2\nu\|\de_y\ff\|^2\\
&-2\nu|k|^2\|\sqrt\alpha\de_y \ff\|^2-2\nu\|\sqrt\alpha\de_{yy}\ff\|^2
-2 k\Re \l i \alpha u'\ff,\de_y\ff\r-2\nu\Re\l\alpha'\de_y\ff,\de_{yy}\ff\r\\
&-2|k|^2\|\sqrt\beta u'\ff\|^2-4\nu k^3\Re\l i \beta u'\ff,\de_y\ff\r+ 4\nu k\Re \l i \beta u'\de_{yy}\ff,\de_y\ff\r
+2\nu k\left[i\beta u'\ff\de_{yy}\ff\,\right]\Big|_{y=0}^1\\
&+2\nu k\Re\l i\beta u'''\ff,\de_y\ff\r+4\nu k \Re\l i\beta'u''\ff,\de_y\ff\r+2\nu k \Re \l i\beta''u'\ff,\de_y\ff\r\\ 
&-2\nu|k|^4\|\sqrt\gamma u'f\|^2-2\nu|k|^2\|\sqrt\gamma u'\de_y\ff\|^2-4\nu|k|^2\Re\l \gamma u'u''\ff,\de_y\ff\r\\
&-2\nu|k|^2\Re\l \gamma'u'\ff,u'\de_y\ff\r.
\end{aligned}
\end{equation}
We will explain how to treat with this term in Section \ref{sec:chan}.
\end{remark}

We now perform soft estimates on the right-hand-side of equation \eqref{eq:der} for the terms 
that do not have a definite sign (with the exception of the term containing $\alpha$, which 
will be treated in a different way in Lemma \ref{lem:err7}), leaving the more difficult and technical ones to Section 
\ref{sec:err}. We will make use of the $\eps$-Young inequality several times without explicit 
mention. In what follows, $C_0\geq 1$ is an absolute constant independent of $\nu$ and $k$.
Firstly, we have
$$
-2\nu\Re\l\alpha'\de_y\ff,\de_{yy}\ff\r\leq 2\nu\left\|\frac{\alpha'}{\sqrt\alpha}\de_y\ff\right\|\|\sqrt\alpha\de_{yy}\ff\|
\leq \frac\nu2\|\sqrt\alpha\de_{yy}\ff\|^2+ C_0\nu\left\|\frac{\alpha'}{\sqrt\alpha}\de_y\ff\right\|^2.
$$
Concerning the first two $\beta$ terms, we have
\begin{equation}\label{eq:hypoerr}
-4\nu k^3\Re\l i \beta u'\ff,\de_y\ff\r
\leq 4\nu |k|^3\left\|\frac{\beta}{\sqrt\alpha} u'\ff\right\|\|\sqrt\alpha\de_y\ff\|
\leq \nu|k|^2\|\sqrt\alpha\de_y \ff\|^2
+ C_0 \nu |k|^4\left\|\frac{\beta}{\sqrt\alpha} u'\ff\right\|^2
\end{equation}
and
$$
 4\nu k\Re \l i \beta u'\de_{yy}\ff,\de_y\ff\r
\leq 4\nu |k|\|\sqrt\alpha\de_{yy}\ff\|\left\|\frac{\beta}{\sqrt\alpha}u'\de_y\ff\right\|
\leq  \frac\nu2\|\sqrt\alpha\de_{yy} \ff\|^2
+C_0\nu|k|^2\left\|\frac{\beta}{\sqrt\alpha}u'\de_y\ff\right\|^2.
$$
Thus, by further restricting \eqref{eq:constr1} to 
\begin{equation}\label{eq:constr2}
\frac{\beta(y)^2}{\alpha(y)}\leq\frac{1}{2C_0} \gamma(y),\qquad  \forall y\in \T,
\end{equation}
we deduce from \eqref{eq:constr1} that
\begin{align*}
-4\nu k^3\Re\l i \beta u'\ff,\de_y\ff\r+4\nu k\Re \l i \beta u'\de_{yy}\ff,\de_y\ff\r
&\leq \nu|k|^2\|\sqrt\alpha\de_y \ff\|^2+\nu |k|^4\|\sqrt\gamma u'\ff\|^2\\
& \quad +\frac\nu2\|\sqrt\alpha\de_{yy} \ff\|^2
+\frac\nu2|k|^2\|\sqrt\gamma u'\de_y\ff\|^2.
\end{align*}
The other three $\beta$-terms can be estimated similarly
as
\begin{align*}
&2\nu k\Re\l i \beta u'''\ff,\de_y\ff\r\leq 2\nu |k|\|\beta u'''\ff\|\|\de_y\ff\|\leq\frac{\nu}{3}\|\de_y\ff\|^2
+C_0\nu|k|^2\|\beta u'''\ff\|^2,\\
&4\nu k \Re\l i\beta'u''\ff,\de_y\ff\r\leq 4\nu |k|\|\beta'u''\ff\|\|\de_y\ff\|\leq \frac{\nu}{3}\|\de_y\ff\|^2+
C_0\nu |k|^2\|\beta'u''\ff\|^2,\\
&2\nu k \Re \l i\beta''u'\ff,\de_y\ff\r\leq 2\nu |k| \| \beta''u'\ff\|\|\de_y\ff\|\leq \frac{\nu}{3}\|\de_y\ff\|^2+
C_0\nu |k|^2 \| \beta''u'\ff\|^2.
\end{align*}
Concerning the $\gamma$-terms, we have
\begin{align*}
-4\nu|k|^2\Re\l \gamma u'u''\ff,\de_y\ff\r
&\leq 4\nu|k|^2\|\sqrt\gamma u''\ff\|\|\sqrt\gamma u'\de_y\ff\|\\
&\leq \frac{\nu}{4} |k|^2\|\sqrt\gamma u'\de_y\ff\|^2
+C_0\nu|k|^2\|\sqrt\gamma u''\ff\|^2,
\end{align*}
and
\begin{align*}
-4\nu|k|^2\Re\l \gamma'u'\ff,u'\de_y\ff\r
&\leq 4\nu|k|^2\left\|\frac{\gamma'}{\sqrt\gamma }u'\ff\right\|\|\sqrt\gamma u'\de_y\ff\|\\
&\leq \frac{\nu}{4} |k|^2\|\sqrt\gamma u'\de_y\ff\|^2
+ C_0\nu|k|^2\left\|\frac{\gamma'}{\sqrt\gamma }u'\ff\right\|^2.
\end{align*}
We now collect all of the above estimates to obtain
\begin{align}
\ddt \Phi \leq &-\nu\|\de_y\ff\|^2-2 k\Re \l i \alpha u'\ff,\de_y\ff\r-2|k|^2\|\sqrt\beta u'\ff\|^2 \nonumber\\
&-2\nu |k|^2\|\ff\|^2-\nu|k|^2\|\sqrt\alpha\de_y \ff\|^2-\nu\|\sqrt\alpha\de_{yy}\ff\|^2
-\nu|k|^4\|\sqrt\gamma u'f\|^2-\nu|k|^2\|\sqrt\gamma u'\de_y\ff\|^2\nonumber\\
 & 
+C_0\Bigg[\nu\left\|\frac{\alpha'}{\sqrt\alpha}\de_y\ff\right\|^2
+ \nu|k|^2\|\beta u'''\ff\|^2
+ \nu |k|^2\|\beta'u''\ff\|^2
+ \nu |k|^2 \| \beta''u'\ff\|^2 \nonumber\\
&
\qquad+ \nu|k|^2\|\sqrt\gamma u''\ff\|^2
+ \nu|k|^2\left\|\frac{\gamma'}{\sqrt\gamma }u'\ff\right\|^2\Bigg] \label{eq:esti1}.
\end{align}
The purpose of the next sections is to provide suitable estimates for the positive terms in the right-hand-side
above. This will require the specification of the precise form of the functions $\alpha,\beta,\gamma$, depending
on the shape of the background flow $u$ and its critical points.

\subsection{Partition of unity}\label{sub:parti}
The dependence of $\alpha,\beta,\gamma$ with respect to $y$ is rather complicated and has to be specific for
each critical point $\bar y_i$, for $i=\{1,\ldots, N\}$. Set
$$
\delta=\min_{i\neq j} \frac{|\bar y_i-\bar y_j|}{8},
$$
and define, for $z\in \RR$, the real functions
$$
\theta(z)=
\begin{cases}
\e^{-1/z}, \quad &z>0,\\
0,\quad &z\leq 0,
\end{cases}
 \qquad \psi(z)=\frac{\theta(z)}{\theta(z)+\theta(1-z)}, 
 \qquad \phi(z)=\psi(z+2)\psi(2-z).
$$
It is not hard to check that $0\leq \phi\leq 1$, and that $\phi(z)=1$ for $|z|\leq 1$ and $\phi(z)=0$ for $|z|\geq 2$.
For each $i=\{1,\ldots, N\}$, define
\begin{align}
\widetilde\phi_i(y)= \phi\left(\frac{y-\bar y_i}{\delta}\right), \label{eq:phii}
\end{align}
where, abusing notation, $\phi$ indicates a proper restriction of the original function to $\T$.
We define
\begin{align*}
\widetilde\phi_0(y)=1-\sum_{i=1}^{N} \widetilde\phi_i(y),
\end{align*}
It is clear by construction that $\left\{\widetilde{\phi_i}\right\}_{i=0}^{N}$ form a partition
of unity. We summarize the properties in the following lemma, state without proof.
\begin{lemma}\label{lem:part1}
The collection of smooth functions $\{\widetilde\phi_i\}_{i=0}^{N}$ satisfies the following properties:
\begin{enumerate}
	\item $\spt(\widetilde\phi_i)\subset [\bar y_i-2\delta,\bar y_i+2\delta] \ $ for every $i\in \{1,\ldots, N\}$;
	\item $\spt (\widetilde\phi_i)\cap \spt (\widetilde\phi_j)=\emptyset\ $ for each $i\neq j\neq 0$;
	\item $\widetilde\phi_i(y)=1\ $  for every $y\in[\bar y_i-\delta,\bar y_i+\delta]$ and every $i\in \{1,\ldots, N\}$;
	\item $\displaystyle \sum_{i=0}^{N} \widetilde\phi_i(y) =1$, for every $y\in \T$.
\end{enumerate}
\end{lemma}
It will also be convenient to group together the critical points that have the same order: to this end,
for each $j\in \{1,\ldots,n_0\}$, let
\begin{equation}\label{eq:tildephi}
\phi_j(y)=\sum_{i\in E_j} \widetilde\phi_i(y),
\end{equation}
where $E_j$ denotes the set of all indices $i$ such that
$$
u^{(\ell)}(\bar y_i)=0,\quad \forall \ell=1,\ldots, j , \qquad u^{(j+1)}(\bar y_i)\neq 0.
$$
When $j=0$, simply put $\phi_0=\widetilde\phi_0$. The properties of $\{\widetilde\phi_i\}_{i=0}^{N}$
transfer naturally to similar ones for $\{\phi_j\}_{j=0}^{n_0}$.

\begin{lemma}\label{lem:part2}
The collection of smooth functions $\{\phi_j\}_{j=0}^{n_0}$ satisfies the following properties:
\begin{enumerate}
	\item $\displaystyle\spt(\phi_j)\subset \bigcup [\bar y_i-2\delta,\bar y_i+2\delta] \ $ for every $j\in \{1,\ldots, n_0\}$,
	where the (disjoint) union is take over all critical points with the same order of vanishing of $u'$;
	\item $\spt (\phi_i)\cap \spt (\phi_j)=\emptyset\ $ for each $i\neq j\neq 0$;
	\item $\phi_j(y)=1\ $  for every $y\in\bigcup [\bar y_i-\delta,\bar y_i+\delta]$ and every $j\in \{1,\ldots, n_0\}$, 
	where the (disjoint) union is take over all critical points with the same order of vanishing of $u'$;
	\item $\displaystyle \sum_{j=0}^{n_0} \phi_j(y) =1$;
	\item for every $\varsigma\in(0,1)$, there exists a constant $C_\varsigma>0$ such that
	\begin{equation}\label{eq:part}
	|\phi_j'(y)|\leq C_\varsigma|\phi_j(y)|^{1-\varsigma},\qquad \forall y\in\T,
	\end{equation}
	for every $j=0,\ldots, n_0$.
\end{enumerate}
\end{lemma}
The last property \eqref{eq:part} is not hard to verify. It certainly holds for $\theta(z)$, whenever $z\in (0,1)$, and for
$\psi$ as well, since
$$
\psi'(z)=\left[\frac{1}{z^2}+\frac{1}{(1-z)^2}\right]\psi(z)\psi(1-z)=\psi'(1-z), \qquad z\in(0,1),
$$
and $\psi(z)\sim \theta(z)$ for $z$ near the origin. Since $\phi$ is essentially defined piecewise in terms of $\psi$,
the estimate \eqref{eq:part} transfers to $\phi$ as well, and, in turn, to each $\phi_j$.

\subsection{Spectral gap estimates} \label{sec:Spec}
One of the main lemmas required to prove our results is a sharp, localized spectral gap-type inequality, 
adapted to the partition of unity defined above.
The fact that it is localized is crucial for controlling the error terms in \eqref{eq:esti1}.

\begin{proposition} \label{LocSpec}
Let $u$ satisfy the hypotheses of Theorem \ref{thm:maindecay} and let $\phi_j$ be as defined above in \eqref{eq:tildephi}.  
The following estimate holds for all $\sigma > 0$ and $f:\T \rightarrow \Complex$ (denoting $f_j = f \sqrt{\phi_j}$),
\begin{align}\label{eq:LocSpec}
\sigma^{\frac{j}{j+1}} \norm{f_j}^2 \lesssim \sigma \norm{\partial_y f_j}^2 + \norm{u' f_j}^2. 
\end{align}
\end{proposition}

\begin{proof}[Proof of Proposition \ref{LocSpec}] 
The proof is a variant of arguments in e.g. \cite{GallagherGallayNier2009}. 
First, by the definition of $\phi_0$ it is clear there holds 
\begin{align*}
\sigma \norm{\partial_y f_0}^2 + \norm{u^\prime f_0}^2 \gtrsim \norm{f_0}^2, 
\end{align*}
as $u^\prime$ is bounded below by a strictly positive number on the support of $\phi_0$. This completes the case $j = 0$.  

Next, for each $j\geq 1$ write $\phi_j = \sum_{i} \widetilde\phi_{i}$ as in \eqref{eq:tildephi}.
As the $\widetilde{\phi_{i}}$ form a partition of unity and the supports of the $\widetilde{\phi_i}$ are disjoint (for $j \geq 1$), it follows that (denoting $\widetilde f_i = f \sqrt{\widetilde\phi_{i}}$) 
\begin{align}
\sigma \norm{\partial_y f_j}^2 + \norm{u' f_j}^2 =  \sum_{i}\sigma \norm{\partial_y \widetilde f_i}^2 + \norm{u' \widetilde f_i}^2, \label{eq:LocSpecDecomp}
\end{align}
where we used that each $\widetilde f_i$ are disjoint, since $j\geq 1$.
Hence, it will suffice to prove the inequality $i$-by-$i$. 

Suppose that $\bar y_i$ is the unique critical point in the support of $\widetilde \phi_i$. Then there is a constant $c_i$ such that 
\begin{align*}
\left(u^\prime(y)\right)^2 \geq c_i^2(y-\bar y_i)^{2 j}, 
\end{align*} 
on the support of $\widetilde \phi_i$. We first claim that there exists a $b_{ij} > 0$ such that
\begin{align}\label{eq:standard}
\norm{\partial_z g}^2 + c_i^2\norm{z^{j} g}^2 \geq b_{ij}\norm{g}^2. 
\end{align}
To see this, notice that the operator $-\de_{zz}+ c_i^2 z^{2j}$ is unbounded on $L^2(\Real)$ and has a compact inverse.
By standard spectral theory \cite{RS80-1}, \eqref{eq:standard} then follows from the spectral gap.
Now, make the re-scaling
\begin{align*} 
(y - \bar y_i) & = \lambda^{-1} z \\ 
g(z) & = \widetilde f_i(\lambda^{-1} z +\bar y_i).  
\end{align*}
Then, 
\begin{align*}
\lambda^{-2} \norm{\partial_y \widetilde f_i}^2 + \lambda^{2j}\norm{c_i (y-\bar y_i)^{j} \widetilde f_i}^2 & \geq b_{ij}\norm{\widetilde f_i}^2 \\ 
\lambda^{2j} \left(\lambda^{-2-2j}\norm{\partial_y \widetilde f_i}^2 + \norm{c_i (y-\bar y_i)^{j} \widetilde f_i}^2\right) & \geq b_{ij}\norm{\widetilde f_i}^2. 
\end{align*}
Making the choice
\begin{align*}
\lambda^{-2-2j} = \sigma,
\end{align*}
implies the desired estimate: 
\begin{align*}
\sigma\|\partial_y \widetilde f_i\|^2 + \|c_i (y-\bar y_i)^{j} \widetilde f_i\|^2 & \geq b_{ij}\sigma^{\frac{j}{j+1}} \|\widetilde f_i\|^2. 
\end{align*}
From \eqref{eq:LocSpecDecomp}, the desired estimate now follows. 
\end{proof} 

\begin{remark}[Global spectral gap] \label{GlobalSpec}
From Proposition \ref{LocSpec}, it is not hard to see that a global version of the spectral-gap inequality holds. Precisely,
let $u$ satisfy the hypotheses of Theorem \ref{thm:maindecay}. 
Then, the following estimate holds for all $\sigma$ sufficiently small and $f:\T \rightarrow \Complex$,  
\begin{align*}
\sigma^{\frac{n_0}{n_0+1}}\norm{f}^2 \lesssim \sigma \norm{\partial_y f}^2 + \norm{u' f}^2. 
\end{align*}
Indeed, as $\set{\phi_j}_j$ defines a partition of unity, for some $C$ (depending only on $u$ through the definition of $\phi_j$ above), there holds   
\begin{align*}
\sigma \norm{\partial_y f}^2 + \norm{u^\prime f}^2 & = \sum_j \sigma \norm{\partial_y f_j}^2 + \norm{u^\prime f_j}^2 - \sigma \norm{\partial_y \phi_j f }^2 - 2\sigma \l\partial_y \phi_j f, \partial_y f_j \r \\ 
& \geq \sum_j\left(\frac{1}{2}\sigma \norm{\partial_y f_j}^2 + \norm{u^\prime f_j}^2\right) - C\sigma\norm{f}^2. 
\end{align*}  
Therefore, by Proposition \ref{LocSpec}, there is another constant $c$ such that 
\begin{align*}
\sigma \norm{\partial_y f}^2 + \norm{u^\prime f}^2 & \geq \sum_j c\sigma^{\frac{j}{j+1}} \norm{f_j}^2 - C\sigma\norm{f}^2 \\ 
& \geq c\sigma^{\frac{n_0}{n_0+1}} \norm{f}^2 - C\sigma\norm{f}^2. 
\end{align*}
The claim then follows for $\sigma$ sufficiently small. 
\end{remark}

\section{Estimates on the error terms}\label{sec:err}
We collect in this section all the estimates needed to properly control the right-hand-side of \eqref{eq:esti1}.
To achieve this, we give a precise expression of the functions $\alpha,\beta,\gamma$ below, which will
have to take into account the different nature of the various critical points of $u$. The error terms in \eqref{eq:esti1}
will be divided in two categories. In Section \ref{sub:primaryerr} we analyze the two main error terms (those
not containing derivatives of $\alpha,\beta,\gamma$), which require a sharp use of the spectral gap-type
inequalities of Proposition \ref{LocSpec}. Section \ref{sub:technicalerr} contains the purely technical estimates, mainly
of the terms containing $\alpha',\beta'$ and $\gamma'$. On the support of these function, $|u'|$ is bounded below
away from zero, and the analysis is somewhat simplified. 

\subsection{Choice of $\alpha,\beta,\gamma$} \label{sec:abgdef}
For the functions $\alpha,\beta,\gamma$ we use the partition of unity of Section \ref{sub:parti} and define
\begin{align*}
\alpha(y)=\sum_{j=0}^{n_0}\eps_{\alpha,j}\alpha_j\phi_j(y), \qquad 
\beta(y)=\sum_{j=0}^{n_0}\eps_{\beta,j}\beta_j \phi_j(y), \qquad 
\gamma(y)=\sum_{j=0}^{n_0}\eps_{\gamma,j}\gamma_j\phi_j(y),
\end{align*}
where
\begin{align}\label{eq:param}
\alpha_j= \frac{\nu^\frac{2}{j+3}}{|k|^\frac{2}{j+3}}, \qquad 
\beta_j= \frac{\nu^\frac{1-j}{j+3}}{|k|^\frac{4}{j+3}}, \qquad 
\gamma_j= \frac{\nu^{-\frac{2j}{j+3}}}{|k|^\frac{6}{j+3}},
\end{align}
and $\eps_{\alpha,j},\eps_{\beta,j},\eps_{\gamma,j}>0$ are small parameters, \emph{independent} of $\nu$ and $k$, 
which will be chosen in Section \ref{sub:mainproof} in order to satisfy the constraints given by \eqref{eq:constr1}-\eqref{eq:constr2} and the other ones derived below. We only preliminary note here that
\begin{equation}\label{eq:constr3}
\beta_j^2=\alpha_j\gamma_j,\qquad \forall j\in\{0,\ldots,n_0\},
\end{equation}
in close analogy with what is prescribed by \eqref{eq:constr1}.
In particular, we will \emph{first} choose $\eps_{\alpha,j},\eps_{\beta,j},\eps_{\gamma,j}>0$ small relative to constants depending only on $u$ and \emph{then} choose the parameter $\kappa_0$
in the restriction $\nu \abs{k}^{-1} \leq \kappa_0$ in the statement of Theorem \ref{thm:maindecay}.  
Notice that 
\begin{align}\label{eq:abgmono} 
\frac{\alpha_j}{\alpha_i}  = \left(\frac{\nu}{\abs{k}}\right)^{\frac{2(i-j)}{(i+3)(j+3)}}, \qquad \frac{\beta_j}{\beta_i}  = \left(\frac{\nu}{\abs{k}}\right)^{\frac{4(i-j)}{(i+3)(j+3)}}, \qquad \frac{\gamma_j}{\gamma_i}  = \left(\frac{\nu}{\abs{k}}\right)^{\frac{6(i-j)}{(i+3)(j+3)}},
\end{align}
which implies that for $\nu \abs{k}^{-1}$ small, $\alpha_{j+1} \gg \alpha_j$, $\gamma_{j+1} \gg \gamma_j$, and $\beta_{j+1} \gg \beta_j$ for all $j$. 
The following result will be used several times below.

\begin{lemma}\label{lem:unwritten}
There exists $\kappa_0 \in (0,1)$ such that  if $\nu|k|^{-1}\leq \kappa_0$ we have (denoting $\ff_j = \ff \sqrt{\phi_j}$), 
\begin{equation}\label{eq:spectrga}
|k|^\frac{2}{j+1}\beta_j^\frac{1}{j+1}\eps_{\beta,j}^\frac{1}{j+1}\nu^\frac{j}{j+1}\|\ff_j\|^2\lesssim \nu \|\de_y\ff_j\|^2+|k|^2\beta_j\eps_{\beta,j}\|u'\ff_j\|^2,\qquad \forall j=0,\ldots n_0
\end{equation}
and
\begin{align}\label{eq:summin}
\sum_{j=0}^{n_0}\left[\nu\|\de_y\ff_j\|^2+|k|^2\eps_{\beta,j}\beta_j\|u'\ff_j\|^2\right]\lesssim \nu\|\de_y\ff\|^2+|k|^2\|\sqrt\beta u'\ff\|^2
\end{align}
\end{lemma}
Notice that the reverse inequality of \eqref{eq:summin} follows trivially  from the triangle inequality.

\begin{proof}
Inequality \eqref{eq:spectrga} follows from \eqref{eq:LocSpec} by appropriately choosing $\sigma$ as
\begin{equation}\label{eq:spectrga2}
\sigma=\frac{\nu}{|k|^2\beta_j\eps_{\beta,j}}=\frac{1}{\eps_{\beta,j}} \left[\nu |k|^{-1}\right]^\frac{2(j+1)}{j+3}. 
\end{equation}
We now turn to \eqref{eq:summin}. Notice that as in Remark \ref{GlobalSpec} we have
\begin{align*}
\nu\|\de_y \ff\|^2+|k|^2\|\sqrt\beta u'\ff\|^2&=\sum_{j=0}^{n_0}\left[\nu\|\de_y\ff_j\|^2+|k|^2\eps_{\beta,j}\beta_j\|u'\ff_j\|^2
-\nu \|\de_y\sqrt{\phi_j}\ff\|^2-2\nu\l \de_y\sqrt{\phi_j}\ff,\de_y\ff_j\r\right]\\
&\geq \sum_{j=0}^{n_0}\left[\frac{\nu}{2}\|\de_y\ff_j\|^2+|k|^2\eps_{\beta,j}\beta_j\|u'\ff_j\|^2\right]
-C\nu\|\ff\|^2.
\end{align*}
By using  \eqref{eq:spectrga},
\begin{align*}
\nu\|\de_y \ff\|^2+|k|^2\|\sqrt\beta u'\ff\|^2
&\geq \frac12\sum_{j=0}^{n_0}\left[\frac{\nu}{2}\|\de_y\ff_j\|^2+|k|^2\eps_{\beta,j}\beta_j\|u'\ff_j\|^2\right] \\ 
& \quad +c\sum_{j=0}^{n_0}\left[\nu^{\frac{j+1}{j+3}}|k|^\frac{2}{j+3} \eps_{\beta,j}^2\|\ff_j\|^2\right] - C\nu \| \ff \|^2\\
&\geq \frac12\sum_{j=0}^{n_0}\left[\frac{\nu}{2}\|\de_y\ff_j\|^2+|k|^2\eps_{\beta,j}\beta_j\|u'\ff_j\|^2\right]+\left[c\nu^{\frac{n_0+1}{n_0+3}}|k|^\frac{2}{n_0+3} -C\nu\right]\|\ff\|^2.
\end{align*}
The result follows by choosing $\kappa_0$ sufficiently small (as in Remark \ref{GlobalSpec}), since
$$
(\nu|k|^{-1})^\frac{2}{n_0+3}\ll 1 \qquad \Rightarrow \qquad  c\nu^{\frac{n_0+1}{n_0+3}}|k|^\frac{2}{n_0+3} -C\nu\geq 0,
$$
and the proof is over.
\end{proof}

\subsection{Primary error terms}\label{sub:primaryerr}
The hardest error terms to control are given by
$$
\nu|k|^2\|\beta u'''\ff\|^2 \qquad\text{and}\qquad\nu|k|^2\|\sqrt\gamma u''\ff\|^2.
$$
Indeed, unlike their derivatives, $\beta$ and $\gamma$ are supported on the whole of $\T$, so
critical points of different order have to be treated differently from each other. We start from the $\beta$ term.

\begin{lemma}\label{lem:err2}
Assume that 
\begin{equation}\label{eq:constr4}
\eps_{\beta,j}\ll 1, \qquad \forall j\in\{0,\ldots,n_0\}.
\end{equation}
There exists $0<\kappa_0\ll1$ such that if $\nu |k|^{-1}\leq\kappa_0$ there holds
$$
\nu|k|^2\|\beta u'''\ff\|^2\leq\frac{1}{14C_0}\left( \nu\|\de_y\ff\|^2+|k|^2\|\sqrt\beta u'\ff\|^2 \right).
$$
\end{lemma}

\begin{proof}
First notice that (denoting $f_j = f \sqrt{\phi_j}$), 
$$
\|\beta u'''\ff\|^2\leq 2\sum_{j}\eps_{\beta,j}^2\beta_j^2 \int_\T |u'''|^2|\ff_j|^2\d y,
$$
so it suffices to work $j$ by $j$. 
We will make use of the spectral gap estimate \eqref{eq:spectrga}
and  distinguish between two cases. If $j\leq 2$, then we simply bound the above sum as
$$
\sum_{j=0}^2\eps_{\beta,j}^2\beta_j^2 \int_\T |u'''|^2|\ff_j|^2\d y \lesssim \sum_{j=0}^2\eps_{\beta,j}^2\beta_j^2\|\ff_j\|^2.
$$
Now, on the support of $\phi_0$, $\abs{u'} \geq c > 0$ for some $c$ and therefore
\begin{align*}
\nu|k|^2\eps_{\beta,0}^2\beta_0^2\|\ff_0\|^2\lesssim \nu|k|^2\eps_{\beta,0}^2\beta_0^2\|u'\ff_0\|^2.
\end{align*}
From \eqref{eq:param} and the restriction $\nu|k|^{-1}\leq \kappa_0$, we obtain
\begin{align*}
\nu|k|^2\eps_{\beta,0}^2\beta_0^2\|\ff_0\|^2=\eps_{\beta,0} (\nu|k|^{-1})^{4/3}\eps_{\beta,0}\beta_0|k|^2\|u'\ff_0\|^2 
\lesssim \eps_{\beta,0}\left[\eps_{\beta,0}\beta_0|k|^2\|u'\ff_0\|^2 \right] .
\end{align*}
For $j=1$ we use the spectral gap \eqref{eq:spectrga} and \eqref{eq:param} to infer that
\begin{align*}
\nu|k|^2 \eps_{\beta,1}^2\beta_1^2\|\ff_1\|^2
& \lesssim (\nu |k|^{-1})^{1/2}\eps_{\beta,1}^{3/2}\left[\nu\|\de_y\ff_1\|^2+\eps_{\beta,1}\beta_1|k|^2\|u'\ff_1\|^2 \right] \\
& \lesssim \eps_{\beta,1}^{3/2}\left[\nu\|\de_y\ff_1\|^2+\eps_{\beta,1}\beta_1|k|^2\|u'\ff_1\|^2 \right],
\end{align*}
while for $j=2$ a similar argument shows that
$$
\nu|k|^2\eps_{\beta,2}^2\beta_2^2\|\ff_2\|^2\lesssim  \eps_{\beta,2}^{5/3}\left[\nu\|\de_y\ff_2\|^2+\eps_{\beta,2}\beta_2
|k|^2\|u'\ff_2\|^2 \right].
$$
Notice that the case $j=2$ is sharp, in the sense that no factor involving positive powers of $\nu |k|^{-1}$ 
appear in the right-hand-side above, so it is necessary to choose $\eps_{\beta,j}$ small.
For $j\geq 3$ the picture is different. In this case, on the support of $\phi_j$, by Taylor's theorem we have that 
$$
|u'''(y)|\lesssim |u'(y)|^{\frac{j-2}{j}},
$$
since $u'$ vanishes to order $j$ and $u'''$ to order $j-2$ at the critical points. Therefore, again reasoning term by term, we apply H\"older's inequality
to obtain
\begin{align*}
\eps_{\beta,j}^2\beta_j^2 \int_\T |u'''|^2|\ff_j|^2\d y&\lesssim\eps_{\beta,j}^2\beta_j^2 \int_\T |u'|^\frac{2(j-2)}{j}|\ff_j|^2\d y\\
&\lesssim \eps_{\beta,j}^2\beta_j^2 \| u' \ff_j\|^\frac{2(j-2)}{j}\|\ff_j\|^\frac{4}{j}\\
&\lesssim\frac{\eps_{\beta,j}^\frac{j+3}{j+1}}\nu  \eps_{\beta,j}\beta_j \| u' \ff_j\|^2
+ \nu^\frac{j-2}{2}\eps_{\beta,j}^\frac{j+4}{j+1}\beta_j^\frac{j+2}{2}\|\ff_j\|^2.
\end{align*}
Thus,
$$
\nu|k|^2\eps_{\beta,j}^2\beta_j^2 \int_\T |u'''|^2|\ff_j|^2\d y\lesssim
\eps_{\beta,j}^\frac{j+3}{j+1}|k|^2 \eps_{\beta,j}\beta_j\|u' \ff_j\|^2+ |k|^2\nu^{j/2}\eps_{\beta,j}^\frac{j+4}{j+1}\beta_j^\frac{j+2}{2}\|\ff_j\|^2.
$$
Hence, from \eqref{eq:param} and \eqref{eq:spectrga}  we have
\begin{align*}
 |k|^2\nu^{j/2}\eps_{\beta,j}^\frac{j+4}{j+1}\beta_j^\frac{j+2}{2}\|\ff_j\|^2
\lesssim \eps_{\beta,j}^\frac{j+3}{j+1}\left[\nu\|\de_y\ff_j\|^2+|k|^2\eps_{\beta,j}\beta_j\|u'\ff_j\|^2\right],
\end{align*}
and therefore we conclude that
$$
\nu|k|^2\eps_{\beta,j}^2\beta_j^2 \int_\T |u'''|^2|\ff_j|^2\d y\lesssim \eps_{\beta,j}^\frac{j+3}{j+1}\left[\nu\|\de_y\ff_j\|^2+
|k|^2\eps_{\beta,j}\beta_j\| u'\ff_j\|^2\right].
$$
We now sum over $j=0,\ldots, n_0$, use Lemma \ref{lem:unwritten} and choose $\eps_{\beta,j}$ small enough.
\end{proof}
The term involving $\gamma$ is controlled by the same quantity, and therefore a constraint 
on the size of $\eps_{\gamma, j}$ in terms of $\eps_{\beta,j}$ needs to be imposed.

\begin{lemma}\label{lem:err3}
Assume that 
\begin{equation}\label{eq:constr5}
 \eps_{\gamma,j} \ll \eps_{\beta,j}^\frac{j}{j+1}, \qquad \forall j\in \{1,\ldots,n_0\}.
\end{equation}
There exists $0<\kappa_0\ll1$ such that if $\nu |k|^{-1}\leq\kappa_0$ there holds
$$
 \nu|k|^2\|\sqrt\gamma u''\ff\|^2\leq\frac{1}{14C_0}\left( \nu\|\de_y\ff\|^2+|k|^2\|\sqrt\beta u'\ff\|^2\right).
$$
\end{lemma}

\begin{proof}
We have (denoting as usual $f_j = f \sqrt{\phi_j}$), 
$$
 \nu|k|^2\|\sqrt\gamma u''\ff\|^2=\nu|k|^2\sum_{j=0}^{n_0}\gamma_j\eps_{\gamma,j}\|u''\ff_j\|^2.
$$
For $j=0$, from \eqref{eq:param} we obtain the trivial estimate (using that $\abs{u'} \geq c > 0$ on the support), 
\begin{align*}
\nu|k|^2\gamma_0\eps_{\gamma,0}\|u''\ff_0\|^2
&\lesssim \nu|k|^2\gamma_0\eps_{\gamma,0}\|\ff_0\|^2\\
&\lesssim \nu|k|^2\gamma_0\eps_{\gamma,0}\|u'\ff_0\|^2\\
&\lesssim (\nu|k|^{-1})^{2/3}\frac{\eps_{\gamma,0}}{\eps_{\beta,0}}\left[|k|^2\beta_0\eps_{\beta,0}\|u'\ff_0\|^2\right].
\end{align*}
In the case $j=1$, \eqref{eq:spectrga} and \eqref{eq:param} entail
$$
\nu|k|^2\gamma_1\eps_{\gamma,1}\|u''\ff_1\|^2\lesssim\nu|k|^2\gamma_1\eps_{\gamma,1}\|\ff_1\|^2
\lesssim \frac{\eps_{\gamma,1}}{\eps_{\beta,1}^{1/2}} \left[\nu \|\de_y\ff_1\|^2+|k|^2\beta_1\eps_{\beta,1}\|u'\ff_1\|^2\right].
$$
For $j\geq 2$, on the support of $\phi_j$ we have (similar to the proof of Lemma \ref{lem:err2} above), 
$$
|u''|\lesssim |u'|^\frac{j-1}{j}.
$$
Consequently,
\begin{align*}
\gamma_j\eps_{\gamma,j}\|u''\ff_j\|^2&\lesssim \gamma_j\eps_{\gamma,j}\|u'\ff_j\|^\frac{2(j-1)}{j}\|\ff_j\|^\frac2j\\
&\lesssim \frac{1}{\nu}\frac{\eps_{\gamma,j}}{\eps_{\beta,j}^\frac{j}{j+1}} \beta_j\eps_{\beta, j}\|u'\ff_j\|^2+ \frac{\eps_{\gamma,j}}{\eps_{\beta,j}^\frac{{j-1}}{j+1}}\frac{\gamma_j^j \nu^{j-1}}{\beta_j^{j-1} }\|\ff_j\|^2,
\end{align*}
and thus
\begin{align*}
\nu|k|^2\gamma_j\eps_{\gamma,j}\|u''\ff_j\|^2\lesssim
\frac{\eps_{\gamma,j}}{\eps_{\beta,j}^\frac{j}{j+1}} |k|^2\beta_j\eps_{\beta, j}\|u'\ff_j\|^2
+ \frac{\eps_{\gamma,j}}{\eps_{\beta,j}^\frac{{j-1}}{j+1}}\frac{\gamma_j^j \nu^{j} |k|^2}{\beta_j^{j-1} }\|\ff_j\|^2.
\end{align*}
We now make use of \eqref{eq:spectrga} on the second term below. Taking into account \eqref{eq:param} one more time,
we have
$$
\frac{\eps_{\gamma,j}}{\eps_{\beta,j}^\frac{{j-1}}{j+1}}\frac{\gamma_j^j \nu^{j} |k|^2}{\beta_j^{j-1} }\|\ff_j\|^2\lesssim
\frac{\eps_{\gamma,j}}{\eps_{\beta,j}^\frac{j}{j+1}} \left[\nu\|\de_y\ff_j\|^2+|k|^2\beta_j\eps_{\beta, j}\|u'\ff_j\|^2\right].
$$
Collecting all of the above estimates, we end up with
$$
 \nu|k|^2\|\sqrt\gamma u''\ff\|^2\lesssim \frac{\eps_{\gamma,0}}{\eps_{\beta,0}}\left[|k|^2\beta_0\eps_{\beta,0}\|u'\ff_0\|^2\right]+\sum_{j=1}^{n_0}\frac{\eps_{\gamma,j}}{\eps_{\beta,j}^\frac{j}{j+1}} \left[\nu\|\de_y\ff_j\|^2+|k|^2\beta_j\eps_{\beta, j}\|u'\ff_j\|^2\right],
$$
so that exploiting Lemma \ref{lem:unwritten} and  \eqref{eq:constr5} the result follows.
\end{proof}

\subsection{Technical error terms}\label{sub:technicalerr}
In this section we turn to the error terms associated with the $y$-dependence of the coefficients; we refer to these errors as ``technical''. 

\begin{lemma}\label{lem:err1}
For $\alpha,\beta,\gamma$ chosen as in Section \ref{sec:abgdef}, for $\nu \abs{k}^{-1}$ sufficiently small, there holds 
\begin{align*}
\nu\norm{\frac{\alpha'}{\sqrt\alpha}\de_y\ff}^2 \leq  \frac{1}{14C_0}\nu \abs{k}^2\norm{\sqrt{\gamma} u' \partial_y f}^2.
\end{align*}
\end{lemma}
\begin{proof} 
First, notice that the integrand is only supported in the region of $y$'s such that $\abs{u'(y)} \geq c > 0$ for some $c > 0$. 
It then suffices to show the pointwise estimate 
\begin{align*}
\abs{\alpha'(y)}^2 \leq \frac{c^2}{14C_0} \gamma(y) \alpha(y) \abs{k}^2  
\end{align*} 
on the support of the integrand. 
From Section \ref{sec:abgdef}, this is equivalent to
\begin{align*}
\abs{\sum_j \eps_{\alpha,j}\alpha_j \phi_j' }^2 \leq \frac{c^2}{14C_0}|k|^2 \left(\sum_j \eps_{\gamma,j}\gamma_j \phi_j \right) \left(\sum_j \eps_{\alpha,j}\alpha_j \phi_j \right).
\end{align*}
At any fixed $y$, only two terms in the summation will be non-zero, $j = 0$ and some other term, say $j = i$. 
Therefore, it suffices to prove:  
\begin{align}
\abs{\eps_{\alpha,0}\alpha_0\phi_0' + \eps_{\alpha,i}\alpha_i\phi_i'}^2 \leq \frac{c^2}{14C_0} \abs{k}^2 \left(\eps_{\gamma,0}\gamma_0 \phi_0 + \eps_{\gamma,i}\gamma_i \phi_i\right) \left(\eps_{\alpha,0}\alpha_0 \phi_0 + \eps_{\alpha,i}\alpha_i \phi_i\right). \label{ineq:Lemapgoal}
\end{align}
We now endeavor to prove \eqref{ineq:Lemapgoal}. 
Note, 
\begin{align*}
\abs{\eps_{\alpha,0}\alpha_0\phi_0' + \eps_{\alpha,i}\alpha_i\phi_i'}^2  & = \abs{\eps_{\alpha,0}\alpha_0 - \eps_{\alpha,i}\alpha_i}^2 \abs{\phi_i'}^2, \\ 
\eps_{\gamma,0}\gamma_0 \phi_0 + \eps_{\gamma,i}\gamma_i \phi_i & = \eps_{\gamma,0}\gamma_0 + (\eps_{\gamma,i}\gamma_i - \eps_{\gamma,0}\gamma_0)\phi_i, \\ 
\eps_{\alpha,0}\alpha_0 \phi_0 + \eps_{\alpha,i}\alpha_i \phi_i & = \eps_{\alpha,0}\alpha_0 + (\eps_{\alpha,i}\alpha_i - \eps_{\alpha,0}\alpha_0)\phi_i. 
\end{align*}
Then, by \eqref{eq:part}, namely 
\begin{align*}
\abs{\phi_i'}^2 \lesssim_\varsigma  \abs{\phi_i}^{2-\varsigma},
\end{align*}
we have 
\begin{align*}
\abs{\eps_{\alpha,0}\alpha_0\phi_0' + \eps_{\alpha,i}\alpha_i\phi_i'}^2 & \lesssim  \left((\eps_{\alpha,i}\alpha_i - \eps_{\alpha,0}\alpha_0)\phi_i\right)^{1-\varsigma} \left(\eps_{\alpha,i}\alpha_i - \eps_{\alpha,0}\alpha_0\right)^{1+\varsigma} \phi_i. 
\end{align*}
For $\nu \abs{k}^{-1}\ll 1$, $\eps_{\alpha,i} \alpha_i \gg \eps_{\alpha,0}\alpha_0$ by \eqref{eq:abgmono}, we have (also choosing $\nu \abs{k}^{-1}$ sufficiently small relative to $\eps_{\alpha,i}\eps_{\gamma,i}^{-1}$),  
\begin{align*}
\left((\eps_{\alpha,i}\alpha_i - \eps_{\alpha,0}\alpha_0)\right)^{1-\varsigma}
 & \lesssim \eps_{\alpha,i}^{1-\varsigma}\left(\frac{\nu}{\abs{k}}\right)^{\frac{2}{i+3}\left(1-\varsigma\right)} \\
 & \lesssim \eps_{\gamma,i}^{1-\varsigma}  \left(\frac{\abs{k}}{\nu}\right)^{\frac{2i}{i+3}\left(1-\varsigma\right)} \\ 
& \lesssim  \abs{k}^{2(1-\varsigma)} \eps_{\gamma,i}^{1-\varsigma} \gamma_{i}^{1-\varsigma}.  
\end{align*}
Then, using that $\gamma_0 = \abs{k}^{-2}$, we get by choosing $\nu \abs{k}^{-1}$ sufficiently small, 
\begin{align*}
\left(\eps_{\alpha,i}\alpha_i - \eps_{\alpha,0}\alpha_0\right)^{\varsigma}\left((\eps_{\alpha,i}\alpha_i - \eps_{\alpha,0}\alpha_0)\phi_i\right)^{1-\varsigma}
& \lesssim \left(\eps_{\alpha_i}\frac{\nu}{\abs{k}}\right)^{\frac{2}{i+3}\zeta} \abs{k}^{2(1-\varsigma)} \eps_{\gamma,i}^{1-\varsigma} \gamma_{i}^{1-\varsigma} \phi_i^{1-\varsigma} \\ 
& \lesssim  \abs{k}^2 \eps_{\gamma,0}^\zeta \gamma_0^\zeta \eps_{\gamma,i}^{1-\varsigma} \gamma_{i}^{1-\varsigma} \phi_i^{1-\varsigma}. 
\end{align*}
This proves \eqref{ineq:Lemapgoal} and hence the lemma follows. 
\end{proof} 

\begin{lemma}\label{lem:err3bis}
For $\alpha,\beta,\gamma$ chosen as described above, there holds for $\nu \abs{k}^{-1}$ sufficiently small,  
\begin{align*}
\nu |k|^2\|\beta'u''\ff\|^2 \leq \frac{1}{14C_0}|k|^2\|\sqrt{\beta} u'\ff\|^2. 
\end{align*}
\end{lemma}
\begin{proof} 
Note that the support of the integrand is on the set where $\abs{u'} \geq c > 0$, for some positive constant $c > 0$ and hence 
it suffices to prove the pointwise bound 
\begin{align*}
\nu\abs{\beta'(y)}^2 \leq \frac{c^2}{14C_0} \beta(y). 
\end{align*}
For any $y$, there is an $i \neq 0$ such that this inequality is equivalent to
\begin{align}
\nu\abs{(\eps_{\beta,i}\beta_i - \eps_{\beta,0}\beta_0)\phi_i'}^2 \leq \frac{c^2}{14C_0}\left(\eps_{\beta,0}\beta_0 + (\eps_{\beta,i}\beta_i - \eps_{\beta,0}\beta_0)\phi_i\right). \label{ineq:errb3isgoal}
\end{align}
By \eqref{eq:abgmono} and \eqref{eq:part}, we have the following for $\nu \abs{k}^{-1}$ sufficiently small,   
\begin{align*}
\nu\abs{(\eps_{\beta,i}\beta_i - \eps_{\beta,0}\beta_0)\phi_i'}^2 & \lesssim \frac{\nu^{1 + \frac{2(1-i)}{i+3}}}{\abs{k}^{\frac{8}{i+3}}} \phi_i \lesssim \eps_{\beta,i}^2\left(\frac{\nu}{\abs{k}}\right)^{\frac{4}{i+3}}
\beta_i \phi_i \\ 
& \lesssim \eps_{\beta,i}\left(\frac{\nu}{\abs{k}}\right)^{\frac{4}{i+3}} \left(\eps_{\beta,0}\beta_0 + (\eps_{\beta,i}\beta_i - \eps_{\beta,0}\beta_0)\phi_i\right). 
\end{align*}
from which the lemma follows by choosing either $\eps_{\beta,i}$ or $\nu \abs{k}^{-1}$ sufficiently small (by \eqref{ineq:errb3isgoal}).  
\end{proof} 

\begin{lemma}\label{lem:err4}
For $\alpha,\beta,\gamma$ chosen as described above, there holds for $\nu \abs{k}^{-1}$ sufficiently small,  
\begin{align*}
\nu |k|^2 \| \beta''u'\ff\|^2 \leq \frac{1}{14C_0}|k|^2\|\sqrt{\beta} u'\ff\|^2. 
\end{align*}
\end{lemma}
\begin{proof} 
The proof follows as in Lemma \ref{lem:err3bis} using that for $\varsigma \in (0,1)$ there holds 
\begin{align*}
\abs{\phi''} \lesssim_\varsigma \abs{\phi}^{1-\varsigma}. 
\end{align*} 
\end{proof} 

\begin{lemma}\label{lem:err6}
For $\alpha,\beta,\gamma$ chosen as described above, there holds for $\nu \abs{k}^{-1}$ sufficiently small,  
\begin{align*}
\nu\abs{k}^2 \norm{\frac{\gamma'}{\sqrt\gamma}u'\ff}^2 \leq  \frac{1}{14C_0}\abs{k}^2\norm{\sqrt{\beta} u' f}^2.
\end{align*}
\end{lemma}

\begin{proof} 
The lemma follows from the pointwise estimate
\begin{align*}
\nu \abs{\gamma'(y)}^2 \leq \frac{1}{14C_0} \beta(y) \gamma(y).
\end{align*}
At any given $y$, there is some $i$ such that this inequality is equivalent to
\begin{align*}
\nu \abs{\eps_{\gamma,0}\gamma_0 \phi_0' + \eps_{\gamma,i}\gamma_i \phi_i'}^2 \leq  \frac{c^2}{14C_0}\abs{\eps_{\beta,0}\beta_0 \phi_0 + \eps_{\beta,i}\beta_i \phi_i}\abs{\eps_{\gamma,0}\gamma_0 \phi_0 + \eps_{\gamma,i}\gamma_i \phi_i}. 
\end{align*}
On the support of the left-hand side we have $\phi_0 = 1 - \phi_i$ and hence we may write this as 
\begin{align}
\nu \abs{\eps_{\gamma,i}\gamma_i - \eps_{\gamma,0}\gamma_0}^2 \abs{\phi_i'}^2 \leq \frac{c^2}{14C_0}\abs{\eps_{\beta,0}\beta_0 + (\eps_{\beta,i}\beta_i - \eps_{\beta,0}\beta_0) \phi_i}\abs{\eps_{\gamma,0}\gamma_0 + (\eps_{\gamma,i}\gamma_i - \eps_{\gamma,0}\gamma_0) \phi_i}. \label{eq:err6gl} 
\end{align}
By \eqref{eq:part}, \eqref{eq:abgmono}, and \eqref{eq:param}  we have for any $\varsigma \in (0,1)$
\begin{align*}
\nu \abs{\eps_{\gamma,i}\gamma_i - \eps_{\gamma,0}\gamma_0}^2 \abs{\phi_i'}^2 & \lesssim_{\varsigma} \nu\eps_{\gamma,i}^2 \gamma_i^2 \left(\phi\right)^{2-\varsigma} \\ 
& \lesssim \eps_{\gamma,i}^2 \left(\frac{\gamma_i}{\gamma_0}\right)^{\varsigma} \gamma_0^{\varsigma} \left(\gamma_i \phi\right)^{1-\varsigma} \left(\nu \gamma_i \phi\right) \\ 
& \lesssim \eps_{\gamma,i}^2 \left(\frac{\gamma_i}{\gamma_0}\right)^{\varsigma} \gamma_0^{\varsigma} \left(\gamma_i \phi\right)^{1-\varsigma} \left(\frac{\nu}{\abs{k}} \right)^{\frac{2}{i+3}} \beta_i \phi. 
\end{align*} 
By \eqref{eq:abgmono} it follows that 
\begin{align*}
\nu \abs{\eps_{\gamma,i}\gamma_i - \eps_{\gamma,0}\gamma_0}^2 \abs{\phi_i'}^2 & \lesssim  \eps_{\gamma,i}^2  \left(\frac{\nu}{\abs{k}} \right)^{\frac{2(1-\varsigma)}{i+3}} \beta_i \phi \gamma_0^{\varsigma} \left(\gamma_i \phi\right)^{1-\varsigma}, 
\end{align*}
from which \eqref{eq:err6gl}, and hence the lemma, follows as $\varsigma \in (0,1)$ so that the exponent is positive and then fixing $\nu \abs{k}^{-1}$ sufficiently small (possibly depending on the parameters $\eps_{\gamma,j},\eps_{\beta,j}$ which are fixed prior to fixing $\nu \abs{k}^{-1}$).  
\end{proof} 

\begin{lemma}\label{lem:err7}
For $\alpha,\beta,\gamma$ chosen as described above, suppose 
\begin{align}\label{eq:constr6}
\eps_{\alpha,j}^2 \leq \frac{1}{196}\eps_{\beta,j}. 
\end{align}
Then there holds, for $\nu \abs{k}^{-1}$ sufficiently small, 
\begin{align*}
\abs{\jap{kiu'\left(\alpha - \frac{\eps_{\alpha,n_0}}{\eps_{\beta,n_0}}\widetilde\lambda_{\nu,k}\beta\right) \ff,\de_y\ff}} \leq  \frac{1}{14}\left(\nu\norm{\de_y \ff}^2 + \abs{k}^2\norm{\sqrt{\beta} u' \ff}^2\right). 
\end{align*}
\end{lemma}
\begin{proof} 
First, by \eqref{eq:abgmono}, we have 
\begin{align}
\widetilde\lambda_{\nu,k}\eps_{\beta,j}\beta_j = \eps_{\beta,j} \alpha_{n_0} \left(\frac{\beta_j}{\beta_{n_0}}\right) =  \eps_{\beta,j} \alpha_{n_0} \left(\frac{\nu}{\abs{k}}\right)^{\frac{4(n_0-j)}{(n_0 + 3)(j+3)}} = \eps_{\beta,j} \alpha_{j} \left(\frac{\nu}{\abs{k}}\right)^{\frac{2(n_0-j)}{(n_0 + 3)(j+3)}}, \label{eq:bjaetc}  
\end{align}
so that for $j < n_0$ this is small for $\nu \abs{k}^{-1}$ sufficiently small, and that the $j = n_0$ term disappears from the term we are trying to estimate. Hence, we have 
\begin{align*}
\alpha - \widetilde\lambda_{\nu,k}\beta = \sum_{j \leq n-1} \left(\eps_{\alpha,j}\alpha_j - \widetilde\lambda_{\nu,k}\eps_{\beta,j}\beta_j\right)\phi_j^2, 
\end{align*}
By \eqref{eq:bjaetc} for $j < n_0$ we have for $\nu \abs{k}^{-1}$ sufficiently small, 
\begin{align*}
\abs{\eps_{\alpha,j}\alpha_j - \widetilde\lambda_{\nu,k}\frac{\eps_{\alpha,n_0}\eps_{\beta,j}}{\eps_{\beta,n_0}}\beta_j} \leq 2\eps_{\alpha_j} \alpha_j. 
\end{align*}
Therefore, 
\begin{align*}
\abs{\jap{kiu'\left(\alpha - \frac{\eps_{\alpha,n_0}}{\eps_{\beta,n_0}}\widetilde\lambda_{\nu,k}\beta\right) \ff,\de_y\ff}} \leq \frac{\nu}{14}\norm{\de_y\ff}^2 + \frac{14}{\nu}\abs{k}^2\norm{\alpha u' \ff}^2. 
\end{align*}
It hence suffices to prove the pointwise inequality 
\begin{align*}
\alpha(y)^2 \leq \frac{\nu}{196} \beta(y).  
\end{align*} 
It follows from the definitions in Section \ref{sec:abgdef} that, 
\begin{align*}
\eps_{\alpha,j}^2\alpha_j^2 & = \eps_{\alpha,j}^2 \nu^{\frac{4}{j+3}} \abs{k}^{-\frac{4}{j+3}} = \left(\frac{\eps_{\alpha,j}^2}{\eps_{\beta,j}}\right) \eps_{\beta,j} \nu^{1 + \frac{1-j}{j+3}} \abs{k}^{-\frac{4}{j+3}} = \left(\frac{\eps_{\alpha,j}^2}{\eps_{\beta,j}}\right) \eps_{\beta,j} \beta_j, 
\end{align*}
and hence, the lemma is proved provided 
\begin{align*}
\eps_{\alpha,j}^2 \leq \frac{1}{196}\eps_{\beta,j}. 
\end{align*}
\end{proof}

\subsection{Proof of Theorem \ref{thm:decay1}}\label{sub:mainproof}
To conclude the proof of Theorem \ref{thm:decay1}, we collect all of the above estimates and check 
that the constraints given by the inequalities \eqref{eq:constr2}, \eqref{eq:constr4}, \eqref{eq:constr5}, \eqref{eq:constr6} are consistent. 
This corresponds to a proper choice 
of $\eps_{\alpha,j},\eps_{\beta,j},\eps_{\gamma,j}$ so that
\begin{equation}\label{eq:con1}
\beta(y)^2\leq\frac{1}{2C_0} \alpha(y)\gamma(y), \qquad \forall y\in \T,
\end{equation}
and
\begin{equation}\label{eq:con2}
\eps_{\beta,j}\ll 1, \qquad \eps_{\gamma,j} \ll \eps_{\beta,j}^\frac{j}{j+1},\qquad
\eps_{\alpha,j}^2 \ll \eps_{\beta,j}. 
\end{equation}
In view of \eqref{eq:constr3}, namely
$$
\beta_j^2=\alpha_j\gamma_j,\qquad \forall j\in\{0,\ldots,n_0\},
$$
we have that
$$
\beta(y)^2=\left[\sum_{j=0}^{n_0}\eps_{\beta,j}\beta_j\phi_j(y)\right]^2\leq 2\sum_{j=0}^{n_0}\eps_{\beta,j}^2\beta_j^2\phi_j(y)^2=
2\sum_{j=0}^{n_0}\eps_{\beta,j}^2\alpha_j\gamma_j\phi_j(y)^2.
$$
On the other hand,
$$
\alpha(y)\gamma(y)=\left[\sum_{j=0}^{n_0}\eps_{\alpha,j}\alpha\phi_j(y)\right]\left[\sum_{j=0}^{n_0}\eps_{\gamma,j}\gamma_j\phi_j(y)\right]
\geq \sum_{j=0}^{n_0}\eps_{\alpha,j}\eps_{\gamma,j}\alpha_j\gamma_j\phi_j(y)^2,
$$
so that \eqref{eq:con1} can be satisfied by imposing that
\begin{equation}\label{eq:con3}
\frac{\eps^2_{\beta_j}}{\eps_{\alpha,j}\eps_{\gamma,j}}\leq \frac{1}{4C_0}.
\end{equation}
We now prove the
decay estimate \eqref{eq:prop2} and fix the constants $\eps_{\alpha,j},\eps_{\beta,j}, \eps_{\gamma,j}$.
From \eqref{eq:esti1} and the lemmata of Section \ref{sec:err} we arrive at
\begin{align}\label{eq:esti2}
\ddt \Phi \leq -\frac{\nu}{2}\|\de_y\ff\|^2- \frac{\eps_{\alpha,n_0}}{\eps_{\beta,n_0}}\widetilde{\lambda}_{\nu,k} k\Re \l i \beta u'\ff,\de_y\ff\r -  \frac{1}{2}|k|^2\|\sqrt\beta u'\ff\|^2
\end{align}
We further use of \eqref{eq:spectrga}, which can be written explicitly  as
$$
\nu^\frac{j+1}{j+3}|k|^\frac{2}{j+3}\eps_{\beta,j}^\frac{1}{j+1}\|\ff_j\|^2\lesssim \nu \|\de_y\ff_j\|^2+|k|^2\beta_j\eps_{\beta,j}\|u'\ff_j\|^2.
$$
Upon summing over $j=0,\ldots,n_0$ and applying Lemma \ref{lem:unwritten} once more, we deduce that
$$
\widetilde{\lambda}_{\nu,k}\|\ff\|^2=\nu^\frac{n_0+1}{n_0+3}|k|^\frac{2}{n_0+3}\|\ff\|^2\lesssim \nu \|\de_y\ff\|^2+|k|^2\|\sqrt\beta u'\ff\|^2.
$$
From \eqref{eq:esti2} we then have for some $c > 0$, 
\begin{align}\label{eq:esti3}
\ddt \Phi \leq -c\widetilde{\lambda}_{\nu,k}\|\ff\|^2- \frac{1}{4}\nu\|\de_y\ff\|^2 - \frac{\eps_{\alpha,n_0}}{\eps_{\beta,n_0}}\widetilde{\lambda}_{\nu,k} k\Re \l i \beta u'\ff,\de_y\ff\r - \frac{1}{4}|k|^2\|\sqrt\beta u'\ff\|^2.
\end{align}
From the expressions of the functions $\alpha,\beta$ and $\gamma$ we also have that
$$
\widetilde\lambda_{\nu,k}\alpha(y)\leq \nu, \qquad \widetilde\lambda_{\nu,k}\gamma(y)\leq \max_{j}\left[\frac{\eps_{\gamma,j}}{\eps_{\beta,j}}\right]\beta(y),
$$
so that 
\begin{align}\label{eq:esti32}
\ddt \Phi \leq -c\widetilde{\lambda}_{\nu,k}\|\ff\|^2- \frac{\widetilde{\lambda}_{\nu,k}}{4}\|\sqrt{\alpha}\de_y\ff\|^2 - \frac{\eps_{\alpha,n_0}}{\eps_{\beta,n_0}}\widetilde{\lambda}_{\nu,k} k\Re \l i \beta u'\ff,\de_y\ff\r - \frac{\widetilde{\lambda}_{\nu,k}}{4}\max_{j}\left[\frac{\eps_{\gamma,j}}{\eps_{\beta,j}}\right]|k|^2\|\sqrt\beta u'\ff\|^2.
\end{align}
We now fix all the constants. Let
$$
\eps_{\gamma,j}=\widetilde \eps\, \eps_{\beta,j}^\frac{j}{j+1}, \qquad \eps_{\alpha,j}=\widetilde \eps\, \eps_{\beta,j},
$$
where $\widetilde \eps \ll 1$ is sufficiently small to satisfy the second two inequalities in \eqref{eq:con2}. Notice that
since $\eps_{\beta,j}<1$, we have
$$
\frac{\eps_{\gamma,j}}{\eps_{\beta,j}}\geq 4\frac{\eps_{\alpha,n_0}}{\eps_{\beta,n_0}}=4\widetilde \eps .
$$
Finally, choose $\eps_{\beta,j}$ small enough so that
$$
\eps_{\beta,j}\ll (\widetilde{\eps})^2\eps_{\beta,j}^\frac{j}{j+1},
$$
which implies that \eqref{eq:con3} is satisfied.
Hence, we deduce the differential inequality
\begin{align}\label{eq:esti4}
\ddt \Phi  \leq - \widetilde{\eps} \,\widetilde{\lambda}_{\nu,k}\Phi,
\end{align}
from which \eqref{eq:prop2} follows. Lastly, from the explicit expressions of $\alpha_j,\beta_j,\gamma_j$ given by \eqref{eq:param}, the equivalence \eqref{eq:prop1} of $\Phi$ with the $H^1$-norm follows, 
concluding the proof of the theorem.

\section{Semigroup decay estimates}\label{sec:L2decay}
In this section, we apply Theorem \ref{thm:decay1} to prove Theorem \ref{thm:maindecay}. 
As we will see, the linear operators we are dealing with are not uniformly sectorial, which makes this procedure non-trivial. 
We recall the following from  \cite{GallagherGallayNier2009}*{Lemma 1.1, part (iii)}.
\begin{lemma}[From \cite{GallagherGallayNier2009}] \label{lem:GGN} 
Let $H:D(H) \rightarrow L^2$ be maximal accretive and sectorial, that is the numerical range is contained in a sector, meaning for some $\delta > 0$ we have  
\begin{align*}
\Theta(H) = \set{\l Hf,f\r \in \Complex: \norm{f} = 1} \subset \set{z \in \Complex: \abs{\arg z} \leq \pi/2 - 2\delta}, 
\end{align*} 
and define 
\begin{align*}
\Psi(H) & = \left(\sup_{\lambda \in \Real} \norm{ \left(H - i\lambda\right)^{-1} }\right)^{-1}. 
\end{align*}
Then, there exists a universal $C > 0$ such that 
\begin{align}
\norm{\e^{-Ht}}_{L^2 \rightarrow L^2} \leq \frac{C}{\tan \delta} \e^{-\frac{1}{2}\Psi(H)t}. \label{ineq:GGNineq}
\end{align}
\end{lemma} 
The proof of Lemma \ref{lem:GGN} follows from representing the semigroup as an integration over a contour encircling the sector containing the numerical range; see \cite{GallagherGallayNier2009}. 
The quantity $\Psi$ is a measure of how large the resolvent is on the imaginary axis, which provides a measure of how close the \emph{pseudo-spectrum} is to the imaginary axis.
Recall that the pseudo-spectrum is basically ``the set where the norm of the resolvent is very large'' -- see e.g. \cite{Trefethen2005} for an in-depth discussion. 
Let
\begin{align*}
L_{k,\nu}  = iku - \nu\left(\partial_{yy} - \abs{k}^2\right),\qquad R_{k,\nu}  = iku - \nu \partial_{yy}, 
\end{align*}
be the linear operators associated with the $k$-th Fourier projections of \eqref{eq:passive} and \eqref{eq:hypopassive}, associated to the linear semigroups 
\begin{align*}
\e^{-tL_{k,\nu}}  = S_\nu(t)P_k, \qquad \e^{-tR_{k,\nu}}  = R_\nu(t)P_k. 
\end{align*}
While for \emph{fixed} $\nu$ and $k$, $L_{k,\nu}$ is sectorial (and maximal accretive), it is not \emph{uniformly} sectorial. 
Indeed,  
\begin{align*}
\l L_{\nu,k}f,f\r = \nu\norm{\partial_y f}^2 + \nu\abs{k}^2\norm{f}^2 + ik \l uf,f\r. 
\end{align*}
Since,
\begin{align*} 
\abs{k\l uf,f\r} & \leq \abs{k} \norm{u}_{L^\infty}\norm{f}^2, 
\end{align*} 
the numerical range is contained in the sector given by angle $\delta$ with $\tan \delta \approx \norm{u}_{L^\infty}^{-1}\abs{k}^{-1}\nu\max\left(1,\abs{k}^2\right)$ (clearly for $\abs{k} \gtrsim 1$ such as the case of $\T$, this reduces to $\tan \delta \approx \norm{u}_{L^\infty}^{-1}\abs{k}\nu$). 
Therefore, Lemma \ref{lem:GGN} gives the following estimates for some $C>0$, (using also that the semigroup is bounded):  
\begin{align*}
&\norm{\e^{-L_{k,\nu}t}}_{L^2 \rightarrow L^2} \leq  \min\left\{1,\frac{C\abs{k}}{\nu\left(1 + k^2\right)}\e^{-\frac{1}{2}\Psi(L_{k,\nu})t}\right\}. 
\end{align*}
For fixed $\nu$ and $k$, $R_{k,\nu}$ is not technically sectorial, however, by Poincar\'e's inequality, it is clear that the only potential problems in $\l R_{k,\nu}f,f \r$ arise when $f$ is close to being a constant (for $f$ far from constant, the numerical range is contained in a sector of angle $\delta$ such that $\tan \delta \approx \nu \abs{k}^{-1}$ analogous to the elliptic case).   
Indeed, given an $f$ with $\norm{f} = 1$, write $a = \l f \r_{D}$ (the $y$-average of $f$) and note by $\int_D u(y) \d y = 0$ we have, 
\begin{align}
\l R_{k,\nu }f,f\r = \nu \norm{\partial_y f}^2 + ik\brak{u(f-a),f-a} + ik\brak{u(f-a),a} + ik \brak{ua,f-a}. 
\end{align}
Since $\|f\|=1$, the latter three terms are bounded by
\begin{align}
\abs{ik\brak{u(f-a),f-a} + ik\brak{u(f-a),a} + ik \brak{ua,f-a}} \lesssim \norm{f-a}\left(1-\norm{f-a}^2\right)^{1/2}. 
\end{align}
Since by Poincar\'e's inequality, there holds $\nu \norm{\partial_y f}^2 \gtrsim \nu \norm{f-a}^2$, it follows that near zero, the numerical range (and hence the spectrum) is contained in a parabola $\abs{\textup{Im} z} \leq C_1 \nu \abs{k}^{-1} \abs{\textup{Re} z}^2$, for some constant $C_1$ depending only on $u$, rather than a sector. 
Hence, Lemma \ref{lem:GGN} technically does not directly apply, however, it is straightforward to check that the proof in \cite{GallagherGallayNier2009}
adapts to this case by drawing the contour $\Gamma$ with (note that $\Psi$ gives an estimate on the real part of the spectrum):
\begin{align}
\Gamma_0 & = \set{z \in \Complex : \textup{Re}\, z = -\frac{1}{2}\Psi, \abs{\textup{Im}\, z} \leq \frac{1}{4}C_1 \nu \Psi^2} \\
\Gamma_1 & = \set{z \in \Complex : \textup{arg}\, z = \pm \abs{k}(2C_1 \nu \Psi)^{-1}, \textup{Re}\, z \leq -\frac{1}{2}\Psi}. 
\end{align}
It then follows that
\begin{align}
&\norm{\e^{-R_{k,\nu}t}}_{L^2 \rightarrow L^2} \leq  \min\left\{1,\frac{C\abs{k}}{\nu \Psi(R_{k,\nu})}\e^{-\frac{1}{2}\Psi(R_{k,\nu})t}\right\}. 
\end{align}
By considering separately times $t \lesssim \Psi^{-1}\left(1-\log\nu\right)$ and $t \gg \Psi^{-1}\left(1-\log\nu\right)$ for $L_{k,\nu}$ 
and times $t \lesssim \Psi^{-1}\left(1+\log\abs{k}\nu^{-1}\Psi^{-1}\right)$ and $t \gg \Psi^{-1}\left(1+\log\abs{k}\nu^{-1}\Psi^{-1}\right)$ for $R_{k,\nu}$, we get the following for some $c > 0$,
\begin{subequations} \label{ineq:semigrp}
\begin{align}
&\norm{\e^{-L_{k,\nu}t}}_{L^2 \rightarrow L^2} \lesssim  \e^{-\frac{c\Psi(L_{k,\nu})}{1+\log \nu^{-1}}  t}, \\ 
&\norm{\e^{-R_{k,\nu}t}}_{L^2 \rightarrow L^2} \lesssim \e^{-\frac{c\Psi(R_{k,\nu})}{1 + \log\nu^{-1} + \log\abs{k} - \log \Psi(R_{k,\nu})} t}.
\end{align}
\end{subequations}
One can take the same bound on both semigroups, possibly after adjusting $c$, as is eventually done in Theorem \ref{thm:maindecay} (only using $\nu \ll k$)

Naturally, the main remaining step is to estimate $\Psi$; for this step the result is the same regardless if one is considering $R_{k,\nu}$ or $L_{k,\nu}$, so let us just consider the latter. 
By the resolvent formula 
\begin{align*}
\left(L_{k,\nu} - z\right)^{-1} = \int_0^\infty \e^{-tL_{\nu,k}} \e^{tz} \d t, 
\end{align*}
applied for $z \in i\Real$, we get 
\begin{align}
\Psi(L_{k,\nu})^{-1} & \leq \int_0^\infty \norm{\e^{-tL_{\nu,k}}}_{L^2 \rightarrow L^2} \d t. \label{eq:ResolvePsi}
\end{align}
It remains to get a good estimate on the right-hand side of \eqref{eq:ResolvePsi}, which is where the augmented energy $\Phi$ will be used.
For each $\nu$ and $k$, we have the standard dissipation estimate, 
\begin{align*}
\nu \int_0^{\nu^{1/2} \abs{k}^{-1/2}}\norm{\partial_y \ff_k(\tau)}^2 \d\tau \leq \norm{P_k \ff_{in}}^2. 
\end{align*}
Therefore, there is some time $\tau_R \in (0,\nu^{1/2}\abs{k}^{-1/2})$ such that 
\begin{align*}
\norm{\partial_y \ff_k(\tau_R)}^2 \leq \frac{\abs{k}^{1/2}}{\nu^{3/2}}\norm{P_k \ff_{in}}^2. 
\end{align*} 
Hence by \eqref{eq:prop1} and that $\Phi$ is monotone decreasing in time by \eqref{eq:prop2}, 
there are exponents $r = r(n_0)$ and $p = p(n_0)$ such that
\begin{align*}
\Phi(\nu^{1/2}\abs{k}^{-1/2}) \leq \Phi(\tau_R) \lesssim \nu^{-r} \abs{k}^{p} \norm{P_k \ff_{in}}^2.    
\end{align*}
By \eqref{eq:prop2}, it follows that 
\begin{align*}  
\Phi(t + \nu^{1/2}\abs{k}^{-1/2}) \leq \e^{-\widetilde\lambda_{\nu,k} t} \Phi(\nu^{1/2} \abs{k}^{-1/2}), 
\end{align*}
where $\widetilde\lambda_{\nu,k}$ is given in \eqref{eq:prop3}. 
As $\norm{P_k \ff(t)}^2 \leq \Phi(t)$, it follows that 
\begin{align*}
\norm{\e^{-(t + \nu^{1/2}\abs{k}^{1/2})L_{\nu,k}}P_k \ff_{in}}^2 \leq \Phi(t+\nu^{1/2} \abs{k}^{-1/2}) \leq \e^{-\widetilde\lambda_{\nu,k} t} \Phi(\nu^{1/2}\abs{k}^{-1/2}) \lesssim \nu^{-r}\abs{k}^p \e^{-\widetilde\lambda_{\nu,k} t} \norm{P_k \ff_{in}}^2.
\end{align*}
We now use this to estimate the right-hand side of \eqref{eq:ResolvePsi}. 
Indeed, there is some $C$ depending only on $u$ such that (using again that $\nu \abs{k}^{-1}$ is sufficiently small), 
\begin{align*}
\Psi(L_{k,\nu})^{-1} \leq \int_0^\infty \norm{\e^{-tL_{\nu,k}}}_{L^2 \rightarrow L^2} \d t & = \left(\int_0^{\nu^{1/2}\abs{k}^{-1/2}} + \int_{\nu^{1/2}\abs{k}^{-1/2}}^\infty\right) \norm{\e^{-tL_{\nu,k}}}_{L^2 \rightarrow L^2} \d t \\ 
& \leq \nu^{1/2}\abs{k}^{-1/2} + \int_{0}^\infty \norm{\e^{-(t + \nu^{1/2}\abs{k}^{-1/2})L_{\nu,k}}}_{L^2 \rightarrow L^2} \d t \\ 
& \leq \nu^{1/2}\abs{k}^{-1/2} + \int_{0}^\infty \min\left(1, C\nu^{-r} \abs{k}^{p} \e^{-\widetilde\lambda_{\nu,k} t}\right) \d t \\
& \lesssim \widetilde\lambda_{\nu,k}^{-1}\left(1 + \log\left(\frac{C\abs{k}^p}{\nu^r}\right)\right) \\ 
& \lesssim \widetilde\lambda_{\nu,k}^{-1}\left(1 + \log \abs{k} + \log \nu^{-1}\right). 
\end{align*}
By similar arguments we get the same estimate on $\Psi(R_{k,\nu})$: 
\begin{align*}
\Psi(L_{k,\nu})^{-1} \lesssim \widetilde\lambda_{\nu,k}^{-1}\left(1 + \log \abs{k} + \log \nu^{-1}\right). 
\end{align*}
Together with the estimates in \eqref{ineq:semigrp}, these estimates on $\Psi$ conclude the proof of Theorem \ref{thm:maindecay}.  
It is clearly unlikely to be possible to remove the logarithmic losses in the rate using this proof; one can perhaps remove one power of the logarithmic losses using slightly more precise micro-local analysis arguments as in \cite{GallagherGallayNier2009}, but this will not correct the logarithmic losses due to the growing numerical range of the operators. 
To deal with the numerical range in the proof of Lemma \ref{lem:GGN}, a more precise control over the contour near the imaginary axis would be required.  

\section{The channel case}\label{sec:chan}
The proof of Theorem \ref{thm:channeldecay} follows a similar general scheme as the proof of Theorem \ref{thm:maindecay}, with one extra set of details to deal with the boundaries. 
In this section we just sketch the main steps which are different in the case of the channel. 
As discussed in Remark \ref{rmk:Channel}, the differential inequality \eqref{eq:der} is slightly altered to \eqref{eq:bdryder} in the case of the channel. 
As in the work above, one derives \eqref{eq:esti1} except with the additional boundary term 
\begin{align} 
2\nu k\left[i\beta u'\ff\de_{yy}\ff\,\right]\Big|_{y=0}^1. \label{def:bdy}
\end{align} 
If $u'$ vanishes on the boundary then this term is automatically eliminated -- this case is much easier than if $u'$ does not vanish on the boundary as then one can then proceed with essentially the same proof that was applied in the $\Torus^2$ case above (see also Remark \ref{rmk:upvanish} below). 
We will here only consider the harder case of $u'(0) \neq 0$ and $u'(1) \neq 0$. 
In this case, we will not be able to control this term with energy estimates, hence instead we will choose $\beta(0) = \beta(1) = 0$. 
To this end, we begin by editing the partition of unity defined in Section \ref{sub:parti}. 
Set 
\begin{align*}
\delta = \min_{i \neq j} \left(\frac{\abs{\bar{y}_i - \bar{y}_j}}{8}, \frac{\abs{\bar{y}_i}}{8}, \frac{\abs{\bar{y}_i - 1}}{8}, \frac{1}{8}\right). 
\end{align*}
Define $\widetilde{\phi}_i$ as in \eqref{eq:phii} for $1 \leq i \leq N$. Now further define, 
\begin{align}\label{def:phibi}
\widetilde{\phi}_{b,0}(y)  = \phi\left(\frac{y}{\delta}\right), \qquad
\widetilde{\phi}_{b,1}(y)  = \phi\left(\frac{y-1}{\delta}\right), 
\end{align}
and then 
\begin{align*}
\widetilde\phi_0(y)= 1- \widetilde{\phi}_{b,0}(y) - \widetilde{\phi}_{b,1}(y)  - \sum_{i=1}^{N}\widetilde\phi_i(y). 
\end{align*} 
Define $\phi_j$ as in \eqref{eq:tildephi}, $\phi_0 = \widetilde{\phi}_0$, and 
\begin{align*}
\phi_b(y) = \widetilde{\phi}_{b,0}(y) + \widetilde{\phi}_{b,1}(y). 
\end{align*}
It is straightforward to verify that this partition of unity satisfies the properties analogous to those outlined in Lemma \ref{lem:part2}. 
For the weights $\alpha,\beta,\gamma$ we now define 
\begin{align*}
&\alpha(y)= \eps_{\alpha,1}\alpha_1 y^2 \widetilde{\phi}_{b,0} + \eps_{\alpha,1} \alpha_1 (1-y^2)\widetilde{\phi}_{b,1} +  \sum_{j=0}^{n_0}\eps_{\alpha,j}\alpha_j\phi_j(y), \\ 
&\beta(y)= \eps_{\beta,1}\beta_1 y^2 \widetilde{\phi}_{b,0} + \eps_{\beta,1} \beta_1 (1-y^2)\widetilde{\phi}_{b,1} + \sum_{j=0}^{n_0}\eps_{\beta,j}\beta_j\phi_j(y), \\ 
&\gamma(y)= \eps_{\gamma,1}\gamma_1 y^2 \widetilde{\phi}_{b,0} + \eps_{\gamma,1} \gamma_1 (1-y^2)\widetilde{\phi}_{b,1} + \sum_{j=0}^{n_0}\eps_{\gamma,j}\gamma_j\phi_j(y), 
\end{align*}
where $\alpha_i,\beta_i,\gamma_i$ are defined as in \eqref{eq:param} and $\eps_{\alpha,j},\eps_{\beta,j},\eps_{\gamma,j}>0$ are small parameters, \emph{independent} of $\nu$ and $k$. 
As the weights now vanish on the boundary, these choices clearly eliminate \eqref{def:bdy} from \eqref{eq:bdryder} and hence it remains to see how to deal with remaining the positive error terms in \eqref{eq:bdryder} (and also to derive the final differential inequality, the analogue of \eqref{eq:esti4}, as  done in Section \ref{sub:mainproof}).  
Naturally, we will need the analogue of Proposition \ref{LocSpec} for the boundary, which we state without proof as it follows analogously. 

\begin{proposition} \label{LocSpec_Halfspace}
Let $u$ satisfy the hypotheses of Theorem \ref{thm:maindecay}, $u^\prime(0) \neq 1$, and let $\widetilde{\phi}_{b,0}$ be as defined above in \eqref{def:phibi}. 
The following estimate holds for all $\sigma > 0$ and $f:[0,\infty) \rightarrow \Complex$ which satisfy $f_y(0) = 0$ (denoting $f_{b} = f \sqrt{\widetilde{\phi}_{b,0}}$),
\begin{align}\label{eq:LocSpecbdy}
\sigma^{1/2} \norm{f_b}^2 \lesssim \sigma \norm{\partial_y f_b}^2 + \norm{y u' f_b}^2. 
\end{align}
By reflection, an analogous statement is true for $\widetilde{\phi}_{b,1}$. 
\end{proposition}
\begin{remark}  \label{rmk:upvanish}
To treat cases where $u'$ vanishes on the boundary we would use a version of Proposition \ref{LocSpec_Halfspace} which yields an inequality analogous to Proposition \ref{LocSpec}.
\end{remark}

Equipped with Proposition \ref{LocSpec_Halfspace}, we may now turn to controlling the boundary contributions to the error terms studied in Section \ref{sec:err}. 
Indeed, via the partition of unity defined by $\phi_j$ and $\phi_b$ as chosen above, the interior terms can be treated as in the case of $\Torus^2$. 
In what follows define 
\begin{align*}
\ff_{b} = \ff \sqrt{\phi_b}. 
\end{align*}

\begin{lemma}\label{lem:err2_bdy}
Assume \eqref{eq:constr4}. Then for $\nu \abs{k}^{-1}$ sufficiently small, there holds
\begin{align*}
\nu|k|^2 \| \beta u'''\ff_b \|^2 \leq \frac{1}{14C_0}|k|^2\|\sqrt\beta u'\ff_b\|^2.
\end{align*} 
\end{lemma}
\begin{proof} 
First note that on the support of $\ff_b$, there is some $c > 0$ such that $\abs{u'} \geq c$.
Second, note that by symmetry, the top and bottom boundaries are essentially the same, so we need only focus on the bottom boundary at $y = 0$. 
Then, the result follows immediately from 
\begin{align*}
\eps_{\beta,1}^2 \nu \abs{k}^2 \beta_1^2 y^4 =  \eps_{\beta,1}^2 \left(\nu \abs{k}^{-1}\right) \abs{k}^2 \beta_1 y^4  \leq \left(\eps_{\beta,1} \nu \abs{k}^{-1}\right) \left(\eps_{\beta,1} \beta_1 \abs{k}^2 y^2\right),  
\end{align*}
and choosing $\nu \abs{k}^{-1}$ small relative to $c$. 
\end{proof} 

Turn next to the analogue of Lemma \ref{lem:err3}. 

\begin{lemma}\label{lem:err3_bdy}
Assume \eqref{eq:constr5} and denote $\ff_{b} = \ff \sqrt{\phi_b}$.
Then, for $\nu \abs{k}^{-1}$ sufficiently small, there holds 
\begin{align*} 
\nu|k|^2\|\sqrt\gamma u''\ff_b\|^2 \leq \frac{1}{14C_0} |k|^2\|\sqrt\beta u'\ff_b\|^2. 
\end{align*} 
\end{lemma}
\begin{proof} 
As in Lemma \ref{lem:err2_bdy} above, on the support of $\ff_b$ there is some $c > 0$ such that $\abs{u'} \geq c$ 
and by symmetry we need only consider the bottom boundary as the top follows similarly. 
Then, the desired inequality follows from 
\begin{align*}
\eps_{\gamma,1} \nu \abs{k}^2 \gamma_1 y = \eps_{\gamma,1}\nu^{1/2} \abs{k}^{1/2}y = \left(\frac{\eps_{\gamma,1}\nu^{1/2}}{\eps_{\beta,1}\abs{k}^{1/2}} \right) \left(\eps_{\beta,1}\abs{k}^2 \beta_1 y\right),
\end{align*}
and choosing $\nu \abs{k}^{-1}$ small relative to $c$, $C_0$, and $\eps_{\gamma,i} \eps_{\beta,1}^{-1}$. 
\end{proof} 

Turn next to the analogue of Lemma \ref{lem:err1}. 

\begin{lemma}\label{lem:err1_bdy}
For $\alpha,\beta,\gamma$ chosen as described above, there holds for $\nu \abs{k}^{-1}$ sufficiently small,  
\begin{align}
\nu\norm{\frac{\alpha'}{\sqrt\alpha}\de_y\ff_b}^2 \leq \frac{1}{14C_0}\left(\nu\norm{\de_y \ff_b}^2 + \nu \abs{k}^2\norm{\sqrt{\gamma} u' \partial_y \ff}^2\right).\label{ineq:err1bdy}
\end{align}
\end{lemma}
\begin{proof} 
First, the region where $\phi_b'(y) \neq 0$ can be treated as in Lemma \ref{lem:err1}; indeed, in this region the additional weight $y$ can be essentially ignored. 
Hence, it suffices to consider the region where $\phi_b'(y) = 0$ (and of course $\phi_b(y) \neq 0$).  
In this region, which is closer to the boundary, the result follows from the observation that 
\begin{align*}
\frac{\alpha'(y)}{\sqrt{\alpha(y)}} = 2\eps_{\alpha,1}^{1/2} \frac{\nu^{1/2}}{\abs{k}^{1/2}},   
\end{align*}
which shows that by choosing $\nu \abs{k}^{-1}$ small, we can absorb the contribution by the first term on the left-hand side of \eqref{ineq:err1bdy}.  
\end{proof} 

Next, the analogue of Lemma \ref{lem:err3bis}. 

\begin{lemma}\label{lem:err3_bdybis}
For $\alpha,\beta,\gamma$ chosen as described above, there holds for $\nu \abs{k}^{-1}$ sufficiently small,  
\begin{align*}
\nu |k|^2\|\beta'u''\ff_b\|^2 \leq \frac{1}{14C_0} |k|^2\|\sqrt{\beta} u'\ff_b\|^2. 
\end{align*}
\end{lemma}
\begin{proof} 
Analogously to Lemma \ref{lem:err1_bdy} above, the region where $\phi'_b(y) \neq 0$ can be treated as in Lemma \ref{lem:err3bis}. 
In the region near the boundary where $\phi'_b(y) = 0$ (and $\phi_b \neq 0$), first note that $\abs{u'} \geq c$ for some $c > 0$. 
Hence, the desired inequality follows from 
\begin{align*}
\nu \abs{k}^2 (\beta'(y))^2 = 4 \eps^2_{\beta,1} \nu \abs{k}^2 \beta_1^2 y^2 = 4 \eps_{\beta,1} \left(\nu \abs{k}^{-1}\right)\abs{k}^2 \beta_1,   
\end{align*}
and choosing $\nu \abs{k}^{-1}$ small relative to $C_0$ and $c$.  
\end{proof} 

Next, we have the analogue of Lemma \ref{lem:err4}. This term requires a more sophisticated treatment than has been used on the previously described boundary error terms, depending now on the spectral gap Proposition \ref{LocSpec_Halfspace}. 

\begin{lemma}\label{lem:err4_bdy}
For $\alpha,\beta,\gamma$ chosen as described above, there holds for $\nu \abs{k}^{-1}$ sufficiently small,  
\begin{align}
\nu |k|^2 \| \beta''u'\ff_b \|^2 \leq \frac{1}{14C_0}\left(\nu\norm{\partial_y \ff_b}^2 + |k|^2\|\sqrt{\beta} u'\ff_b\|^2\right). \label{ineq:betapperr4}
\end{align}
\end{lemma}
\begin{proof} 
In the region near the boundary where $\phi'_b(y) = 0$, we have 
\begin{align*}
\nu \abs{k}^2 (\beta''(y))^2 = 4 \eps^2_{\beta,1} \nu \abs{k}^2 \beta_1^2 = 4 \eps^2_{\beta,1} \nu, 
\end{align*}
which implies 
\begin{align*}
\nu |k|^2 \| \beta''u'\ff_b \mathbf{1}_{\phi'_b = 0} \|^2 \lesssim \eps^2_{\beta,1} \nu \norm{\ff_b}^2. 
\end{align*}
Hence, Proposition \ref{LocSpec_Halfspace} with $\sigma = \nu \beta_1^{-1}\abs{k}^{-2}$, there holds 
\begin{align*}
4 \eps^2_{\beta,1} \nu \norm{\ff_b}^2 & \lesssim \eps^{3/2}_{\beta,1} \nu^{1/2} \beta_1^{-1/2} \abs{k}^{-1} \left(\nu\norm{\partial_y \ff}^2 + \abs{k}^2\norm{\sqrt{\beta}u'\ff}^2\right) \\
 & = \eps^{3/2}_{\beta,1} \nu^{1/2} \abs{k}^{-1/2} \left(\nu\norm{\partial_y \ff}^2 + \abs{k}^2\norm{\sqrt{\beta}u'\ff}^2\right), 
\end{align*}
which is consistent with \eqref{ineq:betapperr4} for $\nu \abs{k}^{-1}$ sufficiently small.
The region where $\phi'_b(y) \neq 0$ can be treated as in Lemma \ref{lem:err4} and is hence omitted for brevity. 
\end{proof} 

Next, the analogue of Lemma \ref{lem:err6}. 

\begin{lemma}\label{lem:err6_bdy}
For $\alpha,\beta,\gamma$ chosen as described above, there holds for $\nu \abs{k}^{-1}$ and
$\eps_{\gamma,1} \eps_{\beta,1}^{-1/2}$ sufficiently small,    
\begin{align*}
\nu\abs{k}^2 \norm{\frac{\gamma'}{\sqrt\gamma(y)}u'\ff_b}^2 \leq \frac{1}{14C_0}\left[\nu \norm{\partial_y \ff_b}^2 +\abs{k}^2\norm{\sqrt{\beta} u' \ff_b}^2\right].
\end{align*}
\end{lemma}
\begin{proof} 
As above, the region where $\phi_b \neq 0$ can be treated as in Lemma \ref{lem:err6}, and is hence omitted for brevity. 
Consider next the region where $\phi'_b = 0$. 
Notice that 
\begin{align*}
\nu \abs{k}^2 \frac{(\gamma'(y))^2}{\gamma(y)} = 4\eps_{\gamma,1}\nu \abs{k}^2 \gamma_1, 
\end{align*}
and therefore we have 
\begin{align*}
\nu\abs{k}^2 \norm{\frac{\gamma'}{\sqrt{\gamma(y)}}u'\ff_b \mathbf{1}_{\phi'_b = 0}}^2 \lesssim \eps_{\gamma,1}\nu \abs{k}^2 \gamma_1\norm{\ff_b}^2.    
\end{align*}
Therefore, by Proposition \ref{LocSpec_Halfspace} with $\sigma = \frac{\nu}{\abs{k} \eps_{\beta,1}}$ we have 
\begin{align*}
\eps_{\gamma,1}\nu \abs{k}^2 \gamma_1\norm{\ff_b}^2 \lesssim \frac{\eps_{\gamma,1}}{\eps_{\beta,1}^{1/2}} \left[\nu \norm{\partial_y \ff_b}^2 + \abs{k}^2 \norm{\sqrt{\beta} u' \ff_b}^2\right].  
\end{align*}
Hence, the desired inequality follows by the stated hypotheses. 
\end{proof} 

Finally, the analogue of Lemma \ref{lem:err7}. 

\begin{lemma}\label{lem:err7_bdy}
For $\alpha,\beta,\gamma$ chosen as described above, there holds, for $\nu \abs{k}^{-1}$ sufficiently small and $\eps_{\alpha,1}^2 \leq \frac{1}{196}\eps_{\beta,1}$,   
\begin{align}
\abs{\jap{\left(\alpha - \frac{\eps_{\alpha,n_c}}{\eps_{\beta,n_c}}\widetilde\lambda_{\nu,k}\beta\right) \de_y \ff,iku'\ff}} \leq \frac{1}{14C_0}\left(\nu\norm{\de_y \ff}^2 + \abs{k}^2\norm{\sqrt{\beta} u' \ff}^2\right). \label{ineq:err7_bdy}
\end{align}
\end{lemma}
\begin{proof} 
If $n_0 \leq 1$, this follows exactly as in Lemma \ref{lem:err7}, as the boundary will make no contribution in that case. 
If $n_0 > 1$, then, as above, first note that, for $\nu \abs{k}^{-1}$ sufficiently small, there holds 
\begin{align*}
\abs{\alpha_1 - \frac{\eps_{\alpha,n_c}}{\eps_{\beta,n_c}}\widetilde \lambda_{\nu,k} \beta_1} \leq 2\alpha_1.
\end{align*}
Therefore, 
\begin{align}
\abs{\jap{ \left(\alpha - \frac{\eps_{\alpha,n_c}}{\eps_{\beta,n_c}} \widetilde \lambda_{\nu,k} \beta\right) \de_y \ff,iku'\ff}} \leq \frac{\nu}{14 C_0}\norm{\de_y \ff}^2 + \frac{14 C_0}{\nu} \abs{k}^2\norm{\alpha u' \ff}^2. \label{whatever}
\end{align}
The non-boundary terms and the region where $\phi_b' \neq 0$ can be treated as in Lemma \ref{lem:err7} above and are hence omitted.      
In the region where $\phi_b' = 0$, we have (by symmetry we can just consider the bottom boundary),  
\begin{align*}
\nu^{-1} \abs{k}^2 \alpha(y)^2 = 4\eps_{\alpha,1}^2 \abs{k} y^4 \leq 4 \eps_{\alpha,1}^2\abs{k}^2 \beta_1 y^2,  
\end{align*}
and hence the boundary contributions of the second term in \eqref{whatever} can be bounded by the second term in \eqref{ineq:err7_bdy} provided that $\eps_{\alpha,1}^2 \leq \frac{1}{196}\eps_{\beta,1}$. 
\end{proof}

This completes the treatment of the error estimates on the positive terms in \eqref{eq:bdryder}. 
Indeed, by checking the constraints and collecting the above error estimates as in Section \ref{sub:mainproof}, we  derive  the analogue of \eqref{eq:esti4}: 
\begin{align}\label{eq:esti4bdy}
\ddt \Phi  \leq - \widetilde \eps \,\widetilde{\lambda}_{\nu,k}\Phi. 
\end{align}
From here, the proof of Theorem \ref{thm:channeldecay} proceeds as in Section \ref{sec:L2decay}.

\section{Kuksin measures}\label{sec:kuksin}
As an application of our results, we consider a stochastically forced drift diffusion equation
of the form
\begin{equation}\label{eq:eq1}
\d f + \left(u\de_x f - \nu \Delta f\right) \d t = \nu^{a/2} \, \Psi\, \d W_t, \qquad f(0)=f_0,
\end{equation} 
where $a>0$ is a fixed parameter, and we analyze the behavior of its invariant measure in the limit $\nu\to 0$.
It is convenient to write the forcing term through the standard Fourier basis, namely
\begin{align*}
\Psi \d W_t = \sum_{(k,j) \in \ZZ^2} \psi_{k,j} e_{k,j}\d W^{k,j}_t, \qquad \psi_{k,j} = \overline{\psi_{-k,-j}},\qquad W^{k,j}_t  = W^{-k,-j}_t,
\end{align*} 
where 
\begin{align*}
e_{k,j} = \frac{1}{4 \pi^2} \e^{-ikx - ijy} 
\end{align*}
and $\{W^{k,j}_t\}_{(k,j) \in \ZZ^2}$ are independent Brownian motions.
Assuming
\begin{equation}\label{eq:noisenorm}
\|\Psi\|^2 =\sum_{(k,j) \in \ZZ^2} |\psi_{k,j}|^2 <\infty,
\end{equation}
the Markov semigroup generated by \eqref{eq:eq1} has a unique invariant measure $\mu_\nu$ on $L^2$ 
and, since
the problem is linear, $\mu_\nu=\mathcal{N}(0,Q_\nu)$, a Gaussian centered at 0 with covariance operator given by 
(see \cite{DZ96})
\begin{equation}\label{eq:covar}
Q_\nu = \nu^{a}\int_0^\infty S_\nu(t)\Psi \Psi^\ast S_\nu(t)^\ast \d t. 
\end{equation}
When $a=1$, the behavior of $\mu_\nu$ for $\nu\to 0$ has been analyzed in detail in \cite{BCZGH15}. Precisely, it is proven in \cite{BCZGH15}*{Theorem 4.6} that 
the sequence of invariant measures $\{\mu_\nu\}_{\nu\in(0,1]}$ 
strongly converges (in the sense that the covariance converges in the strong operator topology), as $\nu\to 0$, to a unique Gaussian measure $\mu_0=\mathcal{N}(0,Q_0)$, invariant for the inviscid equation
$$
\de_t f+u\de_x f=0,
$$
with covariance defined as
\begin{equation}\label{eq:cor0}
Q_0 \varphi= \sum_{j \neq 0} \frac{|\psi_{0,j}|^2}{2|j|^2} \l e_{0,j}, \varphi\r e_{0,j}, \qquad \varphi\in L^2.
\end{equation} 
This confirms the intuition that the inviscid invariant measures constructed 
in this manner  mostly retain information about the long-time dynamics of 
the large scales in the solutions, rather than information about the ``enstrophy'' in the small scales. 
It is clear from \eqref{eq:cor0} that $\mu_0$ is only affected by the $x$-independent modes of the noise. In other words,
the only relevant time-scale for $\mu_0$ is the diffusive one, proportional to $1/\nu$, while the small scales created by the 
enhanced dissipation effects caused by the mixing properties of $u\de_x$ are rapidly annihilated by 
the dissipation on time scales faster than the natural $O(\nu^{-1})$. 

\subsection{Anomalous scalings}
The proof of Theorem \cite{BCZGH15}*{Theorem 4.6} relies on rather general spectral properties of the
operator $u\de_x$, and the quantitative estimate of Theorem \ref{thm:maindecay} is not needed to characterize
the Kuksin measure $\mu_0$ for $a = 1$. However, without a precise decay rate
on the frequencies $k\neq0$ of $S_\nu(t)$, the only treatable value of the parameter $a$ is 1, for any shape of noise,
even in the case  $\psi_{0,j}=0$ for all $j\in \ZZ$.
The interesting point made by Theorem \ref{thm:kuksin} is that, in view of \eqref{eq:maindecay}, 
we are able to modify the balance of diffusivity and noise,
as much as the background shear flow permits.

\begin{proof}[Proof of Theorem \ref{thm:kuksin}]
We will prove that $Q_\nu$, the covariance matrix in \eqref{eq:covar} of $\mu_\nu$, converges to zero in the
operator norm.  Since $\psi_{0,j}=0$ for all $j\in \ZZ$, for any $\varphi \in L^2$ we have 
\begin{align*}
Q_{\nu} \varphi & = \nu^{a}\int_0^\infty S_{\nu}(t) \Psi \Psi^\ast S_{\nu}(t)^\ast \varphi\, \d t \\
& = \nu^{a} \int_0^\infty \sum_{(k,j) \in \ZZ^2:k\neq 0}|\psi_{k,j}|^2 \l e_{k,j}, S_{\nu}(t)^\ast \varphi\r S_{\nu}(t) e_{k,j} \d t \\ 
& = \nu^{a} \int_0^\infty \sum_{(k,j) \in \ZZ^2:k\neq 0} |\psi_{k,j}|^2 \l S_{\nu}(t) e_{k,j},\varphi\r   S_{\nu}(t) e_{k,j} \d t \\   
& =\nu^{a} \int_0^\infty \sum_{(k,j) \in \ZZ^2:k\neq 0} |\psi_{k,j}|^2 \l S_{\nu}(t)P_k e_{k,j},\varphi\r   S_{\nu}(t)P_k e_{k,j} \d t.
\end{align*}
Thanks to \eqref{eq:maindecay}, if $\nu<\kappa_0$ we can estimate the above integral directly. Taking  $\|\varphi\|\leq 1$, we have
\begin{align}
\|Q_{\nu} \varphi\| 
& \leq \nu^{a} \sum_{(k,j) \in \ZZ^2:k\neq 0} |\psi_{k,j}|^2  \int_0^\infty \| S_{\nu}(t)P_k\|^2_{L^2\to L^2} \d t \nonumber\\
& \lesssim  \nu^{a} \sum_{(k,j) \in \ZZ^2:k\neq 0} |\psi_{k,j}|^2  \int_0^\infty  \e^{-2\eps\lambda_{\nu,k}t}\d t \nonumber\\
& \lesssim  \nu^{a} \sum_{(k,j) \in \ZZ^2:k\neq 0} \frac{|\psi_{k,j}|^2}{\eps\lambda_{\nu,k}} \label{eq:esti}.
\end{align}
Since
\begin{equation}\label{eq:rate3}
\lambda_{\nu,k}=\frac{\nu^{\frac{n_0+1}{n_0+3}}|k|^{\frac{2}{n_0+3}}}{(1 + \log |k| + \log\nu^{-1})^2},
\end{equation}
by using \eqref{eq:noisenorm} we obtain
\begin{align*}
\|Q_{\nu}\|_{L^2\to L^2} \lesssim \nu^{a-\frac{n_0+1}{n_0+3}}(1+\log\nu^{-1})^2,
\end{align*}
which vanishes as $\nu\to 0$ provided that
$$
a>\frac{n_0+1}{n_0+3}.
$$
This concludes the proof of the theorem.
\end{proof}

\section*{Acknowledgments}
The authors would like to thank the following people for helpful discussions: Thierry Gallay, Gautam Iyer, Vladim{\'{i}}r \v{S}ver{\'a}k, and Vlad Vicol. 
The work of JB was in part supported by a Sloan Research Fellowship and NSF grant DMS-1413177, the work of MCZ was in part supported by an AMS-Simons Travel Award. 

\appendix 

\section{$H^{-1}$ decay estimates for the inviscid problem}\label{app:mixing}

The result here might be known to some experts, but we include a sketch for completeness.  
The following proof uses duality (as in a similar proof in \cite{LinZeng11}), combined with a basic application of the method of stationary phase for classical oscillatory integrals (see e.g. \cite{BigStein}*{Chapter VIII}). 

\begin{theorem} [Mixing by shear flows in $\Torus^2$] \label{thm:MixingShear}
Let $u \in C^{n_0+2}(\T)$ be such that $u^\prime(y) = 0$ in at most finitely 
many places and suppose that at each critical point, $u^\prime$ vanishes to 
at most order $n_0\geq 1$.
Let $g$ solve the PDE 
\begin{align*}
\partial_t g + u \partial_x g = 0. 
\end{align*}
Then, for all $k \neq 0$ we have 
\begin{align}
\norm{P_k g(t)}_{H^{-1}_y} & \lesssim \jap{kt}^{-1/(n_0+1)}\norm{P_k g(0)}_{H^1_y}, \label{ineq:sheardecay}
\end{align}
where $\jap{kt}=\sqrt{1+(kt)^2}$.
\end{theorem}
\begin{proof}
Denote $P_k g = \ggg_k(t,y)$ and observe that 
\begin{align*}
\ggg_k(t,y) = \e^{-iku(y)t}\ggg_k(0,y). 
\end{align*}
Next, observe that 
\begin{align*}
\norm{\ggg_k(t)}_{H_y^{-1}} = \sup_{\eta \in H_y^1: \norm{\eta}_{H_y^1} = 1} \int_\Torus \ggg_k(t,y) \eta(y)\, \d y.    
\end{align*}
Let $\eta$ be fixed and arbitrary. 
Partition the integral via  
\begin{align*}
\int_\Torus \ggg_k(t,y) \overline{\eta(y)}\, \d y = \sum_j \int_\Torus \e^{-iku(y)t}  \phi_j(y) \ggg_k(0,y)  \overline{\eta(y)}\, \d y = \sum_{j} \mathcal{I}_j,      
\end{align*}
where $\phi_j$ is defined in Section \ref{sub:parti}.
Denote $\psi_j(y) = \phi_j(y) \ggg_k(0,y) \overline{\eta(y)}$. Then, each integral becomes a standard oscillatory integral.  
Bounding $\mathcal{I}_0$ is trivial via integration by parts: 
\begin{align*}
\mathcal{I}_0 & = \int_\Torus \e^{-iku(y)t} \frac{\d}{\d y}\left(\frac{\psi_0(y)}{ikt u'(y)} \right) \d y \lesssim \norm{\psi_0}_{H^1_y}. 
\end{align*} 
Since $H^1_y$ is an algebra in space dimension one,  
\begin{align*}
\norm{\psi_0}_{H^1_y} \lesssim \norm{\ggg_k(0)}_{H^1_y}. 
\end{align*} 
A careful reading of \cite{BigStein}*{Proposition 3, Chapter XIII} 
shows that we may apply the method of stationary phase to deduce 
\begin{align*}
\mathcal{I}_j & \lesssim \jap{kt}^{-1/(j+1)}\left(\norm{\psi_j}_{L^\infty} + \norm{\psi_j}_{H^1_y}\right) \lesssim \jap{kt}^{-1/(j+1)}\norm{\ggg_k(0)}_{H^1_y}.
\end{align*}
The only subtlety is correctly quantifying the loss of regularity -- if one is content with losing more derivatives, then one may apply the proposition directly and even obtain a detailed asymptotic expansion.  
This completes the proof. 
\end{proof}

\bibliographystyle{plain} 
\bibliography{eulereqns,IDnLD}

\end{document}